\documentclass[11pt,leqno]{amsart}
\usepackage{biblatex}
\usepackage{relsize}
\usepackage{color}
\usepackage{mathrsfs}
\usepackage{graphicx}
\usepackage{amssymb,amsmath}
\usepackage[all]{xy}
\usepackage[english]{babel}
\usepackage{enumitem}
\usepackage{eqnarray}
\usepackage{array}
\theoremstyle{plain}
\newtheorem{prop}{Proposition}
\newtheorem{thm}[prop]{Theorem}
\newtheorem{cor}[prop]{Corollary}
\newtheorem{lem}[prop]{Lemma}
\newtheorem{fact}[prop]{Fact}

\newtheorem*{ques}{Question}

\newtheorem*{thmA}{Theorem A}
\newtheorem*{thmB}{Theorem B}
\newtheorem*{corC}{Corollary C}

\newtheorem*{thmD}{Theorem D}

\newtheorem*{thmE}{Theorem E}
\newtheorem*{thmF}{Theorem F}

\newtheorem*{quesG}{Question G}
\newtheorem*{quesH}{Question H}
\theoremstyle{definition}

\theoremstyle{remark}
\newtheorem{rem}[prop]{Remark}
\newtheorem{example}[prop]{Example}

\numberwithin{prop}{section} 
\numberwithin{prob}{section} 
\numberwithin{claim}{prop}
\numberwithin{equation}{section}

\newcommand{\kernel}{\mathrm{ker}}

\newcommand{\Hom}{\mathrm{Hom}}

\newcommand{\iid}{\mathrm{id}}

\newcommand{\spn}{\mathrm{span}}
\newcommand{\stab}{\mathrm{stab}}

\newcommand{\ccd}{\mathrm{cd}}
\newcommand{\ob}{\mathrm{ob}}
\newcommand{\op}{\mathrm{op}}
\newcommand{\Aut}{\mathrm{Aut}}

\newcommand{\idn}{\mathrm{ind}}

\newcommand{\rk}{\mathrm{rk}}

\newcommand{\euF}{\eu{F}}

\newcommand{\ev}{\mathrm{ev}}

\newcommand{\ca}[1]{\mathcal{#1}}
\newcommand{\eu}[1]{\mathfrak{#1}}

\newcommand{\Bor}{\mathrm{Bor}}
\newcommand{\caF}{\mathcal{F}}
\newcommand{\caC}{\mathcal{C}}
\newcommand{\supp}{\mathrm{supp}}
\newcommand{\dd}{\mathrm{d}}
\newcommand{\nor}{|\! |}
\newcommand{\spect}{\mathrm{sp}}

\newcommand{\Z}{\mathbb{Z}}
\newcommand{\Q}{\mathbb{Q}}
\newcommand{\R}{\mathbb{R}}
\newcommand{\F}{\mathbb{F}}
\newcommand{\N}{\mathbb{N}}
\newcommand{\C}{\mathbb{C}}
\newcommand{\E}{\mathbb{E}}

\newcommand{\boG}{\mathbf{G}}
\newcommand{\Iw}{\mathbf{I}}
\newcommand{\CO}{\mathfrak{CO}}
\newcommand{\caO}{\ca{O}}
\newcommand{\caB}{\ca{B}}
\newcommand{\caR}{\ca{R}}
\newcommand{\caS}{\ca{S}}

\newcommand{\QG}{\Q[G]}

\newcommand{\Bi}{\mathbf{Bi}}

\newcommand{\tchi}{\tilde{\chi}}
\newcommand{\trho}{\tilde{\rho}}
\newcommand{\eps}{\varepsilon}
\newcommand{\boh}{\mathbf{h}}
\newcommand{\tr}{\mathrm{tr}}
\newcommand{\der}{\partial}

\newcommand{\euV}{\mathscr{V}}
\newcommand{\euT}{\mathscr{T}}
\newcommand{\euE}{\mathscr{E}}
\newcommand{\gog}{\mathscr{G}}
\newcommand{\eue}{\mathbf{e}}

\newcommand{\caH}{\ca{H}}

\newcommand{\dbl}{[\![}
\newcommand{\dbr}{]\!]}

\newcommand{\End}{\mathrm{End}}
\newcommand{\bcH}{\overline{\caH}}
\newcommand{\bC}{\bar{\C}}

\newcommand{\QGmod}{{}_{\Q[G]}\mathbf{mod}}
\newcommand{\QGdis}{{}_{\Q[G]}\mathbf{dis}}

\newcommand{\FP}{\mathrm{FP}}

\newcommand{\caU}{\mathcal{U}}

\newcommand{\euS}{\mathscr{S}}
\newcommand{\tgamma}{\tilde{\gamma}}

\newcommand{\tzeta}{\tilde{\zeta}}

\newcommand{\argu}{\hbox to 7truept{\hrulefill}}

\newcommand{\ldot}{\, .\,}

\newcommand{\hboG}{\widehat{\boG}}

\title[Rank, Euler characteristic and zeta functions of t.d.l.c.~groups]{The Hattori--Stallings rank, the Euler--Poincar\'e characteristic and zeta functions of totally disconnected locally compact groups}
\author{I.~Castellano} 
\address{Fakultat f\"ur Mathematik, Universit\"at Bielefeld, Bielefeld, Germany}
\email{ilaria.castellano88@gmail.com}
\author{G.~Chinello}
\email{chinellogianmarco@gmail.com}
\author{Th.~Weigel}
\address{Department of Mathematics and Applications, University of Milano Bicocca, Milan, Italy}
\email{thomas.weigel@unimib.it}
\begin{document}
\begin{abstract}
For a unimodular totally disconnected locally compact group $G$ we introduce and study
an analogue of the Hattori--Stallings rank $\trho(P)\in\boh_G$
for a finitely generated projective rational discrete left $\Q[G]$-module $P$.
Here $\boh_G$ denotes the $\Q$-vector 
space of left invariant Haar measures of $G$.
Indeed, an analogue of Kaplansky's theorem holds in this context (cf. Theorem.~A). 
As in the discrete case, using this rank function it is possible to define
a rational discrete Euler--Poincar\'e characteristic $\tchi_G$ whenever $G$ 
is a unimodular totally disconnected locally compact group of type~$\FP_\infty$
of finite rational discrete cohomological dimension. E.g., when $G$ is a discrete group of type~$\FP$, then
$\tchi_G$ coincides with the ``classical'' Euler--Poincar\'e characteristic times the counting measure $\mu_{\{1\}}$. 
For a profinite group $\caO$, $\tchi_\caO$ equals the probability Haar measure $\mu_\caO$ on $\caO$.
Many more examples are calculated explicitly (cf.~Example~\ref{ex:tree} and  Section~\ref{s:euler}). 

In the last section, for a totally disconnected locally
compact group $G$ satisfying an additional finiteness condition,  we introduce and study
a formal Dirichlet series $\zeta_{_{G,\caO}}(s)$ for any compact open subgroup $\caO$.
In several cases it happens that $\zeta_{_{G,\caO}}(s)$
defines a meromorphic function $\tilde{\zeta}_{_{G,\caO}}\colon \C \to\bC$  of the complex plane
satisfying miraculously the identity $\tchi_G=\tzeta_{_{G,\caO}}(-1)^{-1}\cdot\mu_\caO$. 
Here $\mu_{\caO}$ denotes the Haar measure of $G$ satisfying $\mu_{\caO}(\caO)=1$.
\end{abstract}
\date{\today}
\maketitle
\section{Introduction}
\label{s:intro}
For a totally disconnected locally compact (= t.d.l.c.) group $G$ it seems conceivable that the category of discrete left $\QG$-modules $\QGdis$ is a very appropriate category of representations for $G$; see \cite{ACCM23,Cast20,cw:qrat,CCC20,CMW22}). 
In this context there are two rational discrete $\QG$-bimodules which may be interpreted 
as analogues of the group algebra: the 
rational discrete standard bimodule $\Bi(G)$ 
and the $\Q$-vector space of continuous functions $\caC_c(G,\Q)$ with compact support (see \eqref{eq:BG} and \eqref{eq:CcG}). In case that $G$ is unimodular, these two bimodules turn out to be isomorphic but not canonically isomorphic. Indeed, an isomorphism $\psi^{(\caO)}\colon\Bi(G)\to\caC_c(G,\Q)$ can be constructed for any compact open subgroup $\caO$ of $G$. 

For a unimodular t.d.l.c.~group $G$ and a compact open subgroup $\caO$,
let $\mu_{\caO}$ denote the left-invariant (and thus right-invariant) Haar measure on $G$
satisfying $\mu_\caO(\caO)=1$. By
\begin{equation}
\label{eq;defh}
\boh_G=\Q\cdot\mu_{\caO}
\end{equation}
we denote the $\Q$-vector space of Haar measures of $G$ containing all Haar measures which restriction to a compact open subgroup $\caO$ is the probability measure on $\caO$. It naturally contains the positive cone $\boh^+_G=\Q^+\cdot\mu_\caO$ and decomposes as $\boh_G=\boh_G^+\sqcup\{0\}\sqcup\boh^-_G$, where $\boh^-_G=-\boh^+_G$. 
For short we will write $\lambda>0$ for $\lambda\in\boh_G^+$, etc. 

One may define a $\Q$-linear map $\tr\colon \Bi(G)\rightarrow\boh_G$ given by
\begin{equation}
    \tr(x)=\psi^{(\caO)}(x)(1)\cdot\mu_\caO,\quad x\in\Bi(G),
\end{equation}
which satisfies  $\tr(g\cdot x)=\tr(x\cdot g)$ for all $g\in G$ and $x\in\Bi(G)$. Therefore, one may consider the map $\tr$ to be a {\it trace function} on $\Bi(G)$.
This trace function  together with the so-called $\Hom$-$\otimes$ identity for discrete $\QG$-modules (cf.~\cite[\S 4.3]{cw:qrat}) will be used here to define a {\it Hattori--Stallings rank} $\trho(P)\in\boh_G$
for any finitely generated projective discrete left $\QG$-module $P$. The rank function $\trho$ satisfies the following t.d.l.c.~analogue of Kaplansky's theorem (see~\cite{kap:fields,mont}), whose proof absorbs the first part of the paper.
\begin{thmA}
Let $G$ be a unimodular t.d.l.c.~group, and let $P\in \ob(\QGdis)$ be a
finitely generated projective discrete left $\QG$-module. Then
$\trho(P)\geq 0$.
Moreover, $\trho(P)=0$ if, and only if, $P=0$.
\end{thmA}
\begin{rem}
\label{rem:HSperm} For every compact open subgroup $\ca O$ of $G$,
 the Hattori--Stallings rank of the discrete permutation $\QG$-module $\Q[G/\caO]$ has value $1\cdot\mu_\caO$ (see~Pro\-position~\ref{prop:HSperm}).
 \end{rem}
We devote the second part of the paper to introduce and study 
the {\emph{rational Euler--Poincar\'e characteristic}} $\tchi_G$ of 
a unimodular t.d.l.c.~group $G$ of type~$\FP$. 
Recall that the t.d.l.c.~group $G$ is said to be {\em of type~$\FP$} if the trivial left $\QG$-module $\Q$
has a finite projective resolution 
\begin{equation}
\label{eq:proj1}
\xymatrix{0\ar[r]&P_n\ar[r]^-{\der_n}& P_{n-1}\ar[r]^-{\der_{n-1}}&\cdots\ar[r]^-{\der_1} &P_0\ar[r]^-{\eps}&\Q\ar[r]&0}
\end{equation}
in the category $\QGdis$ with each projective discrete left $\QG$-module $P_j$ being finitely generated
(cf.~\cite[\S 4.5]{cw:qrat}). For a unimodular t.d.l.c.~group $G$ of this type
 we define the Euler--Poincar\'e characteristic $\tchi_G$ of $G$  by
\begin{equation}
\label{eq:defep}
\tchi_G=\sum_{k\geq 0} (-1)^k\,\trho(P_k)\in\boh_G.
\end{equation}
As mentioned in \cite{ser:coh} one verifies easily that the value of $\tchi_G$ does not depend on the chosen
projective resolution $(P_\bullet,\der_\bullet)$ of $\Q$ (cf.~\cite[Theorem~2.2.]{stal:centerless}). The definition of the rational Euler--Poincar\'e characteristic $\tilde\chi_G$ for t.d.l.c.~groups traces back to 2015 when it first appeared in the first author's Ph.D.~thesis 
(cf.~\cite[Chapter~2]{Il:PhD}).
Unfortunately, it did not find its way into \cite{cw:qrat}.
 In \cite{pst:betti}  a notion of Euler characteristic was introduced for t.d.l.c.~groups that admit a cocompact topological model. 
The authors related 
 the Euler characteristic to the alternating sum of the $\ell^2$-Betti numbers introduced by  H.D.~Petersen \cite{Pet11}. 
 Note that in this context the Euler characteristic is a rational number depending on the choice of a compact open subgroup $\caO$ of $G$.
 In Subsection~\ref{ss:top mod} we deal with the case of t.d.l.c.~groups admitting a cocompact topological model and show that the two definitions of the Euler characteristic essentially coincide (cf.~Remark~\ref{rem:betti}).
\begin{rem}
\label{rem:EPdis}
Let $G$ be an abstract group which is 
of type~$\FP$ over $\Q$.
Then $G$ is of type~$\FP$, as a discrete t.d.l.c.~group, and $\tchi_G=\chi_G\cdot\mu_{\{1\}}$,
where $\mu_{\{1\}}$ is the counting measure on $G$ and $\chi_G$ is the classical Euler--Poincar\'e characteristic of the  group $G$
(cf.~\cite[pg.~4]{chis:euler} and Remark~\ref{rem:HS disc}).
\end{rem}
\begin{rem}
\label{rem:comptdlc}
Let $\ca O$ be a profinite group. Then $\ca O$ is trivially of type~$\FP$ as compact t.d.l.c.~group (cf.~\cite[Proposition~3.7(a)]{cw:qrat}) and $\tchi_{\ca O}=1\cdot \mu_{\ca O}$, where $\mu_{\ca O}$
is the probability Haar measure on $\ca O$ (cf.  Section \ref{sss:coeul}). 
\end{rem}
The  Euler--Poincar\'e characteristic of an abstract or profinite group is 
an important, but also quite mysterious invariant 
(cf.~\cite[Chap.~IX, \S 8]{brown:coh}, \cite{har:gb}, \cite[\S II.5.4]{ser:gal}, \cite[8.6.14]{nsw:coh}).
It is usually just an integer or a rational number and reflects many quite significant properties.
Surprisingly, in the context of t.d.l.c.~groups, the Euler--Poincar\'e characteristic $\tchi_G$ is no longer just an integer or a rational number, but a rational multiple of a Haar measure, i.e., $\tchi_G\in\boh_G=\Q\cdot\mu_{\caO}$.
In particular, in the best possible situation, one has the following:
\begin{thmB}[\protect{cf.~Theorem~\ref{thm:chi simpl}}]
Let $G$ be a unimodular t.d.l.c.~group acting on a $d$-dimensional abstract
simplicial complex $\Sigma$ with   compact open stabili\-sers.
For every $k\leq d$ let $\Omega_k$ be a set of representatives of the $G$-orbits on the $k$-dimensional simplices of $\Sigma$, and suppose $|\Omega_k|<\infty$. If $\Sigma$ is contractible
then
\begin{equation*}
\tchi_G=\sum_{0\leq k\leq d}(-1)^k\Big(\sum_{\omega\in\Omega_k}1\cdot\mu_{G_{\pm\omega}}\Big),
\end{equation*}
where $G_{\pm\omega}:=\mathrm{stab}_G(\{\omega,-\omega\})$. Here we regard the orbit $G\cdot\omega$ as a subset of the signed $G$-set $\widetilde{\Sigma}_k$ (see~\cite[(A.8)]{cw:qrat}).
\end{thmB}
From Theorem~B one obtains the following equivalent of \cite[\S II.2.9 ($\ast$)]{serre:trees} for t.d.l.c.~groups (cf. $\S$\ref{sss:fundeul} and Proposition~\ref{prop:muless}).
\begin{corC}
Let $\Lambda=(\euV,\euE)$ be a finite connected graph, and let $\gog$ be a 
unimodular\footnote{Here the graph of profinite groups $\gog$ is said to be unimodular, if
$\pi_1(\gog,\Lambda,x_0)$ is a unimodular t.d.l.c.~group.}
 graph of profinite groups
based on $\Lambda$ (cf.~\cite[\S~5.5]{cw:qrat}). Then
\begin{equation*}
\label{eq:fund}
\tchi_{\pi_1(\gog,\Lambda,x_0)}=\sum_{v\in\euV}1\cdot\mu_{\gog_v}-\sum_{\eue\in\euE^g}1\cdot\mu_{\gog_\eue},
\end{equation*}
where $\euE^g\subseteq\euE$ denotes a set of representatives of the $\Z/2\Z$-action on $\euE$
given by edge inversion. In particular, $\tchi_{\pi_1(\gog,\Lambda,x_0)}\leq 0$.
\end{corC}
The final conclusion of Corollary~C holds in a more general context. Indeed, in a forthcoming paper the following more general result will be shown.
\begin{thmD}[\protect{\cite[Theorem~G]{CMW22}}]
Let $G$ be a compactly generated unimodular t.d.l.c.~group satisfying $\ccd_{\Q}(G)=1$. Then
$\tchi_G\leq 0$.
\end{thmD}
By
M.~Dunwoody's theorem (see~\cite[Theorem~1.1]{dun:acc})  the conclusion of Theorem~D is well known for finitely generated discrete groups. Moreover,  a well known result - often attributed to M.~Gromov - states that a discrete group $G$ has a Cayley graph quasi-isometric to a tree if and only if it has a finitely gene\-rated free subgroup of finite index. Therefore, the sign of the Euler--Poincar\'e characteristic of a finitely generated group of rational cohomological dimension 1 is determined by the (non-positive) characteristic of its finite-index free subgroup (see~\cite[Theorem 1]{chis:euler}).

\smallskip 

For a discrete group $G$ every finite symmetric generating system $\Sigma$
defines a length function $\ell\colon G\to\Z$ satisfying $\ell(g)=1$ if, and only if, $g\in\Sigma$.
The growth series of $(G,\Sigma)$ is defined by
\begin{equation*}
\label{eq:grwth}
\gamma_{_{G,\Sigma}}(t)=\sum_{g\in G} t^{\ell(g)}\in\Z\dbl t\dbr.
\end{equation*}
By definition, it is just a formal power series with non-negative integral coefficients.
However, in many significant cases - like for Coxeter groups with their canonical generating system
(see \cite[\S~5.12, Proposition]{hum:cox}) - $\gamma_{_{G,\Sigma}}(t)$ coincides 
with the Taylor expansion in 0 of a rational function 
$\tgamma_{_{G,\Sigma}}(t)\in\C(t)$.

It was shown by J-P.~Serre (see~\cite[\S1.9]{ser:coh}) that for any Coxeter group $(W,\Sigma)$ one has
\begin{equation}
\label{eq:ser}
\chi_{_W}=\frac{1}{\tgamma_{_{W,\Sigma}}(1)}.
\end{equation}
By explicit calculations, one knows that \eqref{eq:ser} also holds for right-angled Artin groups
with their standard generating system (cf.~\cite{atpr:raag}) as well as for surface groups with their standard generating system (cf.~\cite{canwag:grsu}). However, it is also known that the free group on two generators
admits a generating system for which \eqref{eq:ser} does not hold (cf.~\cite{parry:count}).

The work of G.~Harder on arithmetic lattices $G(\caO)$,
where $\caO$ is the ring of integers of a number field $F$ and $G$ is a Chevalley group scheme,
has already revealed a mysterious connection between the {\it Dedekind zeta function} 
$\zeta_{_F}$ of $F$,
and the Euler--Poincar\'e characteristic of $G(\caO)$ (cf.~\cite[Chap.~IX, \S 8]{brown:coh}, \cite{har:gb}).
In a different context, S.~Bouc observed that for a finite group $G$ its
{\it probabilistic zeta function} $P(G,s)$ evaluated in $(-1)$ coincides 
with the reduced Euler--Poincar\'e characteristic of the coset poset of $G$ 
multiplied by $(-1)$ (cf.~\cite{brown:zeta}).
The question we are considering in the last section is the following: is there a function or series
which is miraculously
related to the Euler--Poincar\'e characteristic of a t.d.l.c.~group? 

For a t.d.l.c.~group $G$ and a compact open subgroup $\caO\subseteq G$ let 
$\caR$ denote a set of representatives of the $\caO$-double cosets in 
$G$.
We say that $G$ satisfies the {\emph{double coset property with respect to $\caO$}} if
$\caR(n)=\{r\in \caR\mid \mu_\caO(\caO r\caO)=n\}$
is a finite set for every positive integer $n$. 
Note that if $G$ satisfies the double coset property
with respect to some compact open subgroup $\caO$,  then it satisfies the double coset property with respect to every compact open subgroup (cf.~Proposition~\ref{prop:bdcosgrth}).
This property has the following consequence.
\begin{fact}
\label{fcsg:fin}
A discrete group $G$ has the double coset property if, and only if, $G$ is finite.
\end{fact}
For a t.d.l.c.~group $G$ with the double coset property and any compact open subgroup $\caO$ of $G$ one defines the formal Dirichlet series $\zeta_{_{G,\caO}}(s)$ by
\begin{equation*}
\label{def:dirc}
	\zeta_{_{G,\caO}}(s)=\sum_{n\geq 1} |\caR(n)| 
	n^{-s}=\sum_{r\in\caR}\mu_\caO(\caO 
	r\caO)^{-s}.
\end{equation*}
Our interest in this formal Dirichlet series arose from the following result.
\begin{thmE}[\protect{cf.~Proposition~\ref{prop:zeta build} and Theorem~\ref{thm:PJ}}]
   Let $\Delta$ be a locally finite building of type $(W,S)$ and assume that $\Delta$ has uniform thickness $q+1$.  Let $G$ be a t.d.l.c.~group acting Weyl-transitively on $\Delta$ with compact open stabilizers. Hence,
for every chamber stabilizer $\caO$ in $G$ one has
   \begin{equation}\label{eq:zeta build}
   \zeta_{_{G,\caO}}(s)=\sum_{w\in W}\mu_{\caO}(\caO 
	g_w\caO)^{-s}=\gamma_{_{W,S}}(q^{-s}).
 \end{equation} 
	Moreover, for every subgroup $\caO$ of $G$ which is  either a chamber stabilizer or a spherical parabolic subgroup, one has that
	$\zeta_{_{G,\caO}}(s)$ defines a  meromorphic function 
	$\tzeta_{_{G,\caO}}\colon\C\to \C\cup\{\infty\}$ of the complex plane and 
 \begin{equation}\label{eq:chi zeta}
     \tchi(G)=\frac{1}{\tzeta_{_{G,\caO}}(-1)}\cdot\mu_{\caO}.
     \end{equation}
\end{thmE}
Note that the latter result applies to topological Kac-Moody groups 
of Remy-Ronan type and to  split semisimple simply-connected algebraic groups (cf.~Example~\ref{ex:EPKM} and Remark~\ref{rem:chi build}). In the latter case more can be proved:
\begin{thmF}[cf.~\protect{Theorems~\ref{thm:p-rad}}]
Let $\boG$ be a simple simply-connected Chevalley group scheme and
 $K$  a non-archimedean locally compact field with
finite residue field. Let $G=\boG(K)$
and let $(\widetilde W,\widetilde S)$ be the associated affine Weyl group.
Then,
for every pro-$p$-radical $P^1_
J$ of a parahoric subgroup $P_J$, the series
$\zeta_{_{G,P^1_
J}}(s)$ defines a meromorphic function $\tzeta_{_{G,P^1_
J}}\colon\C\to\C\cup\{\infty\}$
of the complex plane.
Moreover, one has
\begin{equation*}
\label{eq:epid}
\tchi_G=\frac{1}{\tzeta_{_{G,P^1_
J}}(-1)}\cdot \mu_{P^1_
J}.
\end{equation*}
\end{thmF}
\begin{rem}
\label{rem:iwa}
In Subsection~\ref{ss:algebraic} we deal with the special case $\caO=\Iw$, where $\Iw$ denotes  the Iwahori subgroup.
The Bruhat decomposition and Bott's theorem imply that
\begin{equation*}
\tzeta_{_{G,\Iw}}(s)=\tgamma_{_{W,S}}(q^{-s})\prod_{1\leq i\leq n}
\frac{1}{1-q^{-s(d_i-1)}}
\end{equation*}
where $n=\rk(G)$ is the toral rank of $G$, $(W,S)$ is the spherical Weyl group of $G$,
and $(d_i)_{1\leq i\leq n}$ are the degrees of the finite reflection group $(W,S)$.
In particular, for $s\not=0$ the meromorphic function $\tzeta_{_{G,\Iw}}(s)$ satisfies the functional equation
\begin{equation*}
\label{eq:funceq}
\tzeta_{_{G,\Iw}}(-s)=(-1)^{\rk(G)}\cdot\tzeta_{_{G,\Iw}}(s).
\end{equation*}
Moreover, by Theorem~E, one deduces that
 $\tchi_G\in\boh^+(G)$ if $n$ even and $\tchi_G\in\boh^-(G)$ if $n$ odd.
\end{rem}
The results in the last section can be seen as an attempt to give at least a partial answer to the following 
question.

\begin{quesG}
\textup{(a)} For which t.d.l.c.~groups $G$ with the double coset property and for which compact open subgroups $\caO\subseteq G$ does $\zeta_{_{G,\caO}}(s)$ given by \eqref{def:dirc} define a meromorphic function 
$\tzeta_{_{G,\caO}}\colon\C\to\bC$?

\noindent
\textup{(b)} For which unimodular t.d.l.c.~groups $G$ with the double coset property and for which compact open subgroups 
$\caO\subseteq G$
does $\zeta_{_{G,\caO}}(s)$ define a meromorphic function 
$\tzeta_{_{G,\caO}}\colon\C\to\bC$ satisfying $\tchi_G=\tfrac{1}{\tzeta_{_{G,\caO}}(-1)}\cdot\mu_\caO$?
\end{quesG}
\begin{rem}
\label{rem:quesG}
Note that Question~G  has an affirmative answer whenever $G$ is a compact t.d.l.c.~group. Indeed, for $\ca O=G$, $\zeta_{_{G,\ca O}}(s)=\mu_{\ca O}(\caO 1_G\ca O)^{-s}$. Hence, $\tzeta_{_{G,\ca O}}(s)$ is the constant function with value 1 and $\tchi_G=1\cdot\mu_{\ca O}$ (cf. Remark~\ref{rem:comptdlc}).
\end{rem}

\begin{example}
\label{ex:tree}
Let $\euT=\euT_{d+1}$ be a regular locally-finite tree for which every vertex $v\in\euV(\euT)$
has precisely $d+1$ neighbours. We consider the vertices of $\euT$ to be colored by red and blue,
i.e., $\euV(\euT)=\euV(\euT)_r\sqcup\euV(\euT)_b$. All neighbours of a red vertex are blue vertices
and vice versa. Let $A=\Aut(\euT)^\circ$ be the group of color--Preserving automorphisms
of $\euT$. Then $\Aut(\euT)^\circ=\pi_1(\gog,\Lambda,x_0)$, where 
$\Lambda\colon\xymatrix{\bullet_r\ar@{-}[r]^{\eue}&\bullet_b}$. Thus, by \eqref{eq:fund}, one has
\begin{equation}
\label{eq:EPtree}
\tchi_A=\mu_{\gog_{\bullet_b}}+\mu_{\gog_{\bullet_r}}-\mu_{\gog_\eue}=(\tfrac{2}{d+1}-1)\cdot\mu_{\gog_\eue}
=\tfrac{1-d}{1+d}\cdot\mu_{\gog_\eue}.
\end{equation}

\noindent
(a) For $\caO=\gog_v$, where $v\in\{\bullet_r,\bullet_b\}$, one has 
\begin{equation}
\label{eq:treeA2}
\tzeta_{_{A,\gog_v}}(s)=1+\frac{(1+d)^{-s}}{d^s-d^{-s}}.
\end{equation}
For every $g\in \caR$, by the orbit-stabiliser theorem, the integer  $\mu_{\ca O}(\ca Og\caO)=|\caO\colon\ca O\cap g\caO g^{-1}|$ is the cardinality of the $\caO$-orbit of the vertex $gv$. In particular, if $g\neq 1$ then
\begin{equation*}
\mu_{\ca O}(\ca Og\caO)=|\{w\in\euV(\euT)\mid \mathrm{dist}_{\ca T}(v,w)=2k\}|=(d+1)d^{2k-1}
\end{equation*}
where  $2k=\mathrm{dist}_{\ca T}(v,gv)$ and $k\geq1$. Notice that each $gv$ is at even distance from $v$ since the tree has two colors and $g$ is color--Preserving.
Therefore,
\begin{eqnarray*}\displaystyle
\zeta_{_{A,\gog_v}}(s)=&\displaystyle\mu_{\caO}(\caO)^{-s}+\sum_{g\in\caR\setminus\{1\}}\mu_\caO(\caO g\caO)^{-s}&=\displaystyle 1+\sum_{k\geq 1}\big((d+1)d^{2k-1}\big)^{-s}=
\\
=&\displaystyle 1+(1+d)^{-s}d^{s}\sum_{k\geq1}(d^{-2s})^k.&
\end{eqnarray*}
Within the radius of convergence, the sum defines a meromorphic function $\tzeta_{A,\gog_v}$ of the complex plane given by
\begin{equation}\displaystyle
1+\frac{(1+d)^{-s}}{d^{-s}}\Big(\frac{1}{1-d^{-2s}}-1\Big)
=1+\frac{d^{-s}(1+d)^{-s}}{1-d^{-2s}}=1+\frac{(1+d)^{-s}}{d^s-d^{-s}},
\end{equation}
i.e., \eqref{eq:treeA2} holds.
 Moreover,
\begin{equation}
\tzeta_{_{A,\gog_v}}(-1)=1+\frac{1+d}{d^{-1}-d}=\frac{d^{-1}+1}{d^{-1}-d}=\frac{1}{(d^{-1}-1)d}
=\frac{1}{1-d}
\end{equation}
As $\mu_{\gog_v}=\tfrac{1}{d+1}\cdot\mu_{\gog_\eue}$ this implies that 
(b) of Question G  holds for $(A,\gog_v)$.

\noindent
(b) For $\caO=\gog_\eue$ one obtains 
\begin{equation}
\label{eq:treeA1}
\tzeta_{_{A,\gog_\eue}}(s)=\frac{1+d^{-s}}{1-d^{-s}},
\end{equation}
and hence by \eqref{eq:EPtree}, the relation (b) of Question G also holds for $(A,\gog_\eue)$. Indeed, for every $g\in \caR$, the integer  $\mu_{\ca O}(\ca Og\caO)=|\caO\colon\ca O\cap g\caO g^{-1}|$ is the cardinality of the $\caO$-orbit of the edge $g\eue$. Let $\euT_r$ and $\euT_b$ be the two $d$-regular rooted trees obtained from $\euT$ after removing the edge $\eue$. If $g\neq 1$, the edge $g\eue$ (together with its own $\caO$-orbit) belongs either to $\euT_r$ or $\euT_b$. Accordingly, the set of $\ca O$-double coset representatives decomposes as $$\ca R=\{1\}\sqcup\ca R_r\sqcup\ca R_b.$$ In particular, if $g\in\ca R_\bullet$ (with $\bullet\in\{r,b\}$), then $\mu_{\ca O}(\ca Og\caO)=d^k$,
where  $k\geq 1$ is the level of the rooted tree $\euT_\bullet$ containing $g\eue$.
Therefore,
\begin{align*}\displaystyle
\zeta_{_{A,\gog_v}}(s)&=\displaystyle\mu_{\caO}(\caO)^{-s}+\sum_{g\in\caR_r}\mu_\caO(\caO g\caO)^{-s}+\sum_{g\in\caR_b}\mu_\caO(\caO g\caO)^{-s}=\\
&=\displaystyle 1+2\sum_{k\geq 1}\big(d^{k}\big)^{-s}=\displaystyle 1+2\sum_{k\geq 1}\big(d^{-s}\big)^{k}.
\end{align*}
Within the radius of convergence, the sum defines a meromorphic function given by
\begin{equation}\displaystyle
1+2\Big(\frac{1}{1-d^{-s}}-1\Big)
=1+\frac{2d^{-s}}{1-d^{-s}}=\frac{1+d^{-s}}{1-d^{-s}},
\end{equation}
i.e., \eqref{eq:treeA1} holds.

A deep investigation of the double-coset zeta functions for groups acting on trees has been performed by B.~Marchionna in~\cite{M24}.
\end{example}

Although Question~G might be very difficult to answer in general, the following more specific question might be the starting point for further investigations.

\begin{quesH}
For $G$  p-adic analytic and $\caO\subseteq G$ compact open subgroup does part (a) of Question G has an affirmative answer in general?
\end{quesH}


\subsection*{Acknowledgement} We thank Bianca Marchionna for some useful comments on a first draft of our manuscript. The first named
author was supported by the Deutsche Forschungsgemeinschaft (DFG, German Research Foundation) – SFB-TRR 358/1 2023
– 491392403. 
The first and third authors are members of the Gruppo Nazionale per le Strutture Algebriche, Geometriche e le loro
Applicazioni (GNSAGA), which is part of the Istituto Nazionale di Alta Matematica (INdAM).

\section{Idempotents in $C^\ast$-algebras}
\label{s:Cidem}

\subsection{$C^\ast$-algebras}
\label{ss:cstar}
A $\C$-algebra $A$ together with a {\it norm} $\nor \argu\nor\colon A\to\R_0^+$
and an {\it involution} $\argu^\ast\colon A^{\op}\to A$ is said to be a {\it $C^\ast$-algebra} 
(cf.~\cite[\S 1.1]{sakai:cast}) if
\begin{enumerate}[label={(C\arabic*})]
\item$(A,\nor \argu\nor)$ is a Banach space;
\item$\nor x y\nor\leq \nor x\nor \nor y\nor$ for all $x,y\in A$;
\item $\argu^\ast$ is an isomorphism of $\R$-algebras satisfying 
$$(\lambda x)^\ast= \bar{\lambda} x^\ast,\quad\forall\lambda\in\C, x\in A;$$
\item $\nor x^\ast x\nor=\nor x\nor^2$ for all $x\in A$.
\end{enumerate}

In a $C^\ast$-algebra $A$ one has $\nor x\nor=\nor x^\ast\nor$ for all $x\in A$
(cf.~\cite[Lemma~1.1.6]{sakai:cast}).
If the $C^\ast$-algebra $A$ contains a unit $1\in A$, it is called {\it unital}.
For such a $C^\ast$-algebra $A$, the {\it spectrum} of an element $a\in A$ is defined by
\begin{equation}
\label{eq:spec}
\spect_A(a)=\{\,\lambda\in\C\mid a-\lambda 1\ \text{is not invertible}\,\}.
\end{equation}
Elements $a\in A$ satisfying $a=a^\ast$ are called {\it self-adjoint},
and elements $a\in A$ which are self-adjoint satisfying $\spect_A(a)\subseteq\R^+_0$
are said to be {\it positive}.
By I.~Kaplansky's theorem, one knows that an element $a\in A$ is positive if, and only if,
there exists $b\in A$ such that $a=b^\ast b$ (cf.~\cite[Thm.~1.4.4]{sakai:cast}).
As a consequence, in a unital $C^\ast$-algebra all elements of the form
$1+b^\ast b$, $b\in A$, are invertible.
We put $A^+=\{\,a\in A\mid a\ \text{positive}\,\}$.
\subsection{$C^\ast$-algebras acting on a Hilbert space}
\label{ss:CHil}

It is well known, that if $(V,\langle.,.\rangle)$ is a complex Hilbert space, the algebra $\caB(V)$ of
bounded linear operators from $V$ to $V$ together with the adjoint map $\argu^\ast\colon \caB(V)^{\op}
\to\caB(V)$ and the operator norm $\nor\argu\nor\colon \caB(V)\to\R_0^+$
is a $C^\ast$-algebra (cf.~\cite[\S 1.15]{sakai:cast}).

Let $A$ be a unital $C^\ast$-algebra, and let $V$ be a complex Hilbert space
which is also a left $A$-module. Then we say that $A$ is {\it acting on the complex Hilbert space} $V$ if the
induced map $\sigma\colon A\to\caB(V)$ is a $\ast$-homomorphism of $C^\ast$-algebras satisfying
\begin{equation}
\label{eq:Hact}
\nor \sigma(a)\nor=\nor a\nor
\end{equation}
for all $a\in A$. Obviously, \eqref{eq:Hact} implies that $\sigma$ must be injective.

\subsection{Trace functions and special elements}
\label{ss:special}
A {\it trace function} of a unital $C^\ast$-algebra $A$ is defined as a $\C$-linear function $\tau\colon A\to\C$ satisfying $$\tau(ab)=\tau(ba),\quad\text{for all $a,b\in A$.}$$

Suppose $A$ is acting on the complex Hilbert space $(V,\langle.,.\rangle)$.
An element $\delta\in V$ is said to be {\it special} if the following are satisfied:
\begin{enumerate}[label={(S\arabic*})]
\item\label{eq:S1} $\langle\delta,\delta\rangle=1$;
\item\label{eq:S2} the canonical map $\rho_\delta\colon A\to V$, $\rho_\delta(a)=a.\delta$, $a\in A$, is injective;
\item\label{eq:S3} $\langle \rho_\delta(ab),\delta\rangle=\langle \rho_\delta(ba),\delta\rangle$ for all $a,b\in A$.
\end{enumerate}
A special element $\delta\in V$ defines the trace function 
$\tau_\delta\colon A\to\C$ by  $$\tau_\delta(a)=\langle \rho_\delta(a),\delta\rangle,\quad\text{for every $a\in A$.}$$
\begin{prop}
\label{prop:spec}
Let $A$ be a unital $C^\ast$-algebra acting on the complex Hilbert space $(V,\langle.,.\rangle)$,
and let $\delta\in V$ be a special element. Then
\begin{itemize}
\item[(a)] $\tau_\delta(1)=1$;
\item[(b)] $\tau_\delta(ab)=\tau_\delta(ba)$ for all $a,b\in A$;
\item[(c)] $\tau_\delta(a)\in\R^+_0$ for all $a\in A^+$;
\item[(d)] For $a\in A^+$ one has
$\tau_\delta(a)=0$ if, and only if, $a=0$.
\end{itemize} 
\end{prop}
\begin{proof}
(a) is a consequence of property \ref{eq:S1} of a special element, and (b) is a consequence of \ref{eq:S3}.
(c) By Kaplansky's theorem, for $a\in A^+$ there exists $b\in A$ such that $a=b^\ast b$. Hence
\begin{equation}
\label{eq:specprop}
\tau_\delta(a)=\langle (b^\ast b).\delta,\delta\rangle=\langle b.\delta,b.\delta\rangle\in\R^+_0.
\end{equation}
(d) By \eqref{eq:specprop}, $\tau_\delta(a)=0$ if, and only if $\rho_\delta(b)=b.\delta=0$.
Hence, by property \ref{eq:S2}, this is equivalent to $b=0$ and, therefore, $a=0$.
\end{proof}
A trace function $\tau\colon A\to\C$ of a unital $C^\ast$-algebra $A$ satisfying property (c) of Proposition~\ref{prop:spec}
is called {\it positive}; if it satisfies property (d), then it is called {\it faithful}.

\subsection{Matrices over a unital $C^*$-algebra} Let $A$ be a unital $C^*$-algebra acting on the complex Hilbert space $(V,\langle.,.\rangle)$. There is a natural structure on the set $M_n(A)$ of $n\times n$ matrices over $A$ that makes it a unital $C^*$-algebra acting on a complex Hilbert space. The algebraic structure on $M_n(A)$ corresponds to the usual one and the involution $\argu^\ast\colon M_n(A)^{\op}\to M_n(A)$ is defined by 
$$[m_{jk}]^*=[m^*_{kj}],\quad j,k\in\{1,\ldots,n\},$$
for every matrix $M=[m_{jk}]\in M_n(A)$. Moreover, the involutive algebra $M_n(A)$ acts on the $n$-fold direct sum $V\oplus\cdots\oplus V$ via the usual matrix action on column vectors. Therefore, the linear injective $*$-homomorphism
$$\bar\sigma\colon M_n(A)\to \ca B(V\oplus\cdots\oplus V)$$
defines the norm on $M_n(A)$ by $||M||=||\bar\sigma(M)||,\ M\in M_n(A)$.
\begin{fact}\label{fact:matrix tr}
 Let $\tau\colon A\to\C$ be a trace function on the unital $C^*$-algebra $A$. The map
 $$\bar\tau\colon M_n(A)\to\C,\quad [m_{ij}]\mapsto \sum_{i=1}^n \tau(m_{ii}),$$
 is a  trace function on the unital $C^*$-algebra $M_n(A)$. Moreover, $\bar\tau$ is faithful and positive whenever $\tau$ is.
\end{fact}


\subsection{Faithful positive traces and idempotents}
\label{ss:idem}
The following theorem - which is basically due to
M.~Burger and A.~Valette - is essential for our purpose.
Although they did not state it in this general form, the proof
of the following is identical to their proof of \cite[Thm.~2.1]{bv:idem} which relies on the following lemma due to I.~Kaplansky.
\begin{lem}[\protect{\cite[Theorem~26]{kap}}]\label{lem:kap}
For every idempotent $e$ of a unital $C^*$-algebra there exists a self-adjoint idempotent $f$ such that $ef=f$ and $fe=e$.
\end{lem}
\begin{proof}
Let $e$ be an idempotent and set
\begin{equation}
\label{eq:alma1}
z=1+(e^\ast-e)^\ast(e^\ast-e).
\end{equation}
In particular, $z$ is self-adjoint and invertible.
Put $f=e e^\ast z^{-1}$.
A straightforward computation shows that $ze=ee^*e=ez$. Hence,
$z$ commutes with $e$, and thus also with $e^\ast$.
This implies that $f=ee^*z^{-1}=z^{-1}ee^*=f^\ast$. Moreover,
\begin{equation}
\label{eq:alma2}
f^2=ee^\ast z^{-1}ee^\ast z^{-1}=z^{-1}(ee^\ast e)e^\ast z^{-1}=z^{-1}(ze)e^\ast z^{-1}=f,
\end{equation}
 i.e., $f$ is a self-adjoint idempotent. Clearly, $ef=f$. On the other hand, $fe=ee^*z^{-1}e=(ee^*e)z^{-1}=(ez)z^{-1}=e$.
\end{proof}

\begin{thm}[M.~Burger \& A.~Vallette]
\label{thm:traceidem}
Let $A$ be a unital $C^\ast$-algebra with a positive and faithful trace function $\tau\colon A\to\C$. For every  $e\in A$ idempotent, one has  $\tau(e)\in[0,\tau(1)]\subset\R^+_0$. Moreover, $\tau(e)=0$ if and only if $e=0$. 
\end{thm}

\begin{proof} 
Let $e$ be an idempotent of $A$. By Lemma~\ref{lem:kap}, there is $f\in A$ 
self-adjoint idempotent such that $ef=f$ and $fe=e$. As $\tau$ is a trace 
function, this yields
$\tau(f)=\tau(ef)=\tau(fe)=\tau(e)$.
In particular, $\tau(e)\in\R^+_0$ because $f=f^2=ff^*\in A^+$ and  $\tau$ is 
positive. Moreover, if $\tau(e)=0$ then $\tau(f)=0$ and, as $\tau$ is faithful, 
this yields $f=0$ and thus $e=fe=0$.
Finally, since $1-f=(1-f)^*(1-f)\in A^+$, one has $\tau(1-f)\geq0$ and so 
$\tau(f)\leq \tau(1)$.
\end{proof}




\section{Hecke algebras for t.d.l.c.~groups}
\label{s:ftdlc}
Throughout this section $G$ will denote a 
t.d.l.c.~group, and $\caO\subseteq G$
will be a compact open subgroup of $G$.


\subsection{Function spaces}
\label{ss:funcomsupp}
Any ring $\E\in\{\,\Z,\Q,\R,\C\,\}$  will be always considered to be endowed with the discrete topology. By $\euF(G,\E)$ we denote  the $\E$-module of functions
from $G$ to $\E$. This $\E$-module is canonically a left/right $\E[G]$-bimodule,
where for $f\in \euF(G,\E)$, $g,x\in G$, one has
\[(g \ldot f)(x)= f(g^{-1}x),\qquad (f\ldot g)(x)=f(x g^{-1}).\]
By
\[\begin{aligned}
\euF(G,\E)^{\caO}&=\{\,f\in\euF(G,\E)\mid \forall \omega\in\caO:\  f\ldot\omega=f\,\},\text{ and}\\
{}^{\caO}\euF(G,\E)^{\caO}&=\{\,f\in\euF(G,\E)\mid \forall \omega_1,\omega_2\in\caO:\  
\omega_1\ldot(f\ldot\omega_2)=f\,\},\\
\end{aligned}\]
we denote the $\E$-module of right and the $\E$-module of left and right fixed points with respect to $\caO$, respectively.
Any function of either of these $\E$-modules is locally constant and hence continuous.
By $\caC_c(G,\E)\subseteq\euF(G,\E)$ we denote
the $\E$-submodule of {\it continuous functions with compact support}  and put
\begin{equation*}
\begin{aligned}
\ca C_c(G,\E)^{\caO}&=\ca C_c(G,\E)\cap\euF(G,\E)^{\caO}\ \text{and}\\ 
{}^{\caO}\caC_c(G,\E)^{\caO}&=\caC_c(G,\E)\cap{}^{\caO}\euF(G,\E)^{\caO}.
\end{aligned}
\end{equation*}
\begin{rem}\label{rem:loc_const}
As $\E$ is discrete, every function $f\in\caC_c(G,\E)$ must be locally constant: for $x\in G$ there exists a compact open subgroup
$\caO_x$ such that $f$ is constant on $x\caO_x$; since $\supp(f)$ is compact, there exists a finite set
$\Omega\subseteq \supp(f)$ such that $\supp(f)\subseteq \bigcup_{x\in\Omega} 
x\caO_x$. 
In particular, $f\in \caC_c(G,\E)^{\caO}$ for $\caO=\bigcap_{x\in\Omega}\caO_x$ and hence
\[\textstyle{\caC_c(G,\E)=\bigcup_{\caO\subseteq G}\caC_c(G,\E)^{\caO},}\]
where the union is running over all compact open subgroups of $G$.
\end{rem}
Given $S\subseteq G$, let $I_S(\argu)\colon G\to\E$ be the characteristic function defined by  $I_S(x)=1$ if $x\in S$ and $I_S(x)=0$ otherwise.
For all $g\in G$, one has $g\ldot I_S=I_{gS}$ and $I_S\ldot g=I_{Sg}$. Moreover, we define
$$[g\dbr:= I_{g\caO}\in \caC_c(G,\E)^{\caO}\quad\text{and}\quad \dbl g\dbr:= I_{\caO g\caO}\in {}^{\caO}\caC_c(G,\E)^{\caO},$$
whenever $\ca O$ is a compact open subgroup of $G$.

In view of the remark above, the following fact is straightforward.

\begin{fact}
\label{fact:perm}
Let $G$ be a t.d.l.c.~group, let $\caO$ be a compact open subgroup of $G$, and 
let $\E\in\{\,\Z,\Q,\R,\C\,\}$. 
\begin{itemize}
\item[(a)] The $\E$-linear map $\phi\colon \E[G/\caO]\longrightarrow\caC_c(G,\E)^{\caO}$,
$\phi(y\caO)=[y\dbr$, $y\in G$, is an isomorphism of discrete left $\E[G]$-modules.
\item[(b)] The $\E$-linear map $\psi\colon \E[\caO\backslash G/\caO]\longrightarrow{}^{\caO}\caC_c(G,\E)^{\caO}$,
$\psi(\caO y\caO)=\dbl y\dbr$, $y\in G$, is an isomorphism of $\E$-modules.
\end{itemize}
\end{fact}

For $f\in\euF(G,\E)$ we will denote by $f^\ast\in\euF(G,\E)$
the function given by
\begin{equation}\label{eq:star}
f^\ast(x)=\overline{f(x^{-1})},\qquad x\in G,
\end{equation}
where $\overline{\cdot}\colon\C\to\C$ is complex conjugation,
i.e., $\argu^\ast\colon\euF(G,\E)\to\euF(G,\E)$ is skew-linear and
$\argu^{\ast\ast}=\iid_{\euF(G,\E)}$.


\subsection{Integrable functions}
\label{ss:intfun}
By $\mu\colon\Bor(G)\to\R^+_0\cup\{\infty\}$ we denote 
a left-invariant
Haar measure of $G$.
Here $\Bor(G)$ denotes the {\it set of Borel sets}, i.e., the
$\sigma$-algebra generated by all open subsets of $G$. Let $\Delta\colon 
G\to\Q^+$ be the {\it modular function} of $G$: $\mu(Sg)=\Delta(g)\mu(S)$ for 
all $S\in\Bor(G)$ and $g\in G$. The integral defines the $\E$-linear map
$\eps_\mu\colon\caC_c(G,\E)\longrightarrow \C$ given  by
\[\eps_\mu(f)=\int_G f(\omega)\,\dd\mu(\omega),\qquad
f\in \caC_c(G,\E),\]
which depends on  $\mu$ and, by Remark~\ref{rem:loc_const}, satisfies the following:
\begin{align}
\eps_\mu(g.f)&=\int_G f(g^{-1}\omega)\,\dd\mu(\omega)=\eps_\mu(f),\label{eq:int1}\\
\eps_\mu(f.g)&=\int_G f(\omega g^{-1})\,\dd\mu(\omega)=\Delta(g)\eps_\mu(f),\label{eq:int2}\\
\int_G f(\omega^{-1})\,\dd\mu(\omega)&=\int_G f(\omega)\Delta(\omega^{-1})\,\dd\mu(\omega)\label{eq:int3}
\end{align}
for all $g\in G$ and $f\in \caC_c(G,\E)$. 

A function $f\in\euF(G,\E)$ is called {\it integrable} if 
\[\eps_\mu(|f|)=\int_G |f(\omega)|\,\dd\mu(\omega)\]
exists and is less than $\infty$, and {\it square integrable} if $f^2$ is integrable.
We denote by $L^1(G,\E)\subseteq\euF(G,\E)$ and $L^2(G,\E)\subseteq\euF(G,\E)$ the $\E$-vector spaces of 
{\it integrable} and {\it
square integrable functions}, respectively. We also put
\begin{equation}\label{eq:L1ast}
\begin{aligned}
L^1_\ast(G,\E)&=\{\,f\in L^1(G,\E)\mid f^\ast\in L^1(G,\E)\,\}\\
&=\{\,f\in L^1(G,\E)\mid f\cdot\Delta^\ast\in L^1(G,\E)\,\}.
\end{aligned}
\end{equation}
(cf.~\eqref{eq:star},\eqref{eq:int3}). By definition, the integral can be extended to
\[\eps_\mu\colon L^1(G,\E)\longrightarrow \C,
\qquad
\eps_\mu(f)=\int_G f(\omega)\,\dd\mu(\omega),\qquad
f\in L^1(G,\E),\]
satisfying \eqref{eq:int1} and \eqref{eq:int2}. Moreover, if $f\in L^1_\ast(G,\E)$ then also \eqref{eq:int3} holds.
By $L^k(G,\E)^{\caO}$ and ${}^{\caO}L^k(G,\E)^{\caO}$, $k\in\{\,1,2,1_\ast\,\}$, we denote the intersection
with $\euF(G,\E)^{\caO}$ and ${}^{\caO}\euF(G,\E)^{\caO}$, respectively.


\subsection{The Hilbert space $L^2(G,\C)^{\caO}$}
\label{ss:hilb}
From now on, we will choose the Haar measure $\mu_\caO$ that satisfies $\mu_\caO(\caO)=1$.
On $L^2(G,\C)^{\caO}$ one has a skew-symmetric sesqui-linear form
\begin{equation}\label{eq:sequilin}
\begin{gathered}
\langle \argu\, ,\argu\rangle\colon L^2(G,\C)^{\caO}\times L^2(G,\C)^{\caO}\longrightarrow\C,\\
\langle f,h\rangle=\int_G f(\omega)\cdot \overline{h(\omega)}\,\dd\mu_\caO(\omega),\qquad f,h\in L^2(G,\C)^{\caO},
\end{gathered}
\end{equation}
making $(L^2(G,\C)^{\caO},\langle.,.\rangle)$ a {\it Hilbert space}. Indeed, if $f\in L^2(G,\C)^{\caO}$, and 
$x\in\supp(f)$, then $x\caO\subseteq\supp(f)$. Hence two functions $f_1,f_2\in L^2(G,\C)^{\caO}$ which
differ on a set of measure $0$ must be equal. Note that by \eqref{eq:int1}, the 
skew-symmetric sesqui-linear form $\langle\argu\, ,\argu\rangle$ is $G$-invariant, i.e.,
one has
\[\langle g\ldot f,g\ldot h\rangle=\langle f,h\rangle\]
for all $f,h\in L^2(G,\C)^{\caO}$ and $g\in G$. By $\nor\argu\nor_2\colon L^2(G,\C)^{\caO}\to\R_0^+$ we will denote the Hilbert norm associated
to $(L^2(G,\C)^{\caO},\langle.,.\rangle)$. The following fact is straight\-forward (cf.~\cite[\S III.5.2, Cor., p.263]{bou:top}).

\begin{fact}
\label{fact:suppL2}
Let $G$ be a t.d.l.c.~group, let $\caO$ be a compact open subgroup,
and let $\caS\subseteq G$ be a set of representatives for $G/\caO$.
\begin{itemize}
\item[(a)] $\{\,[u\dbr\mid u\in\caS\,\}$ is an orthonormal Hilbert basis of the Hilbert space $L^2(G,\C)^{\caO}$, i.e.,
\[\textstyle{L^2(G,\C)^{\caO}=\Big\{\,\sum_{u\in\caS} \lambda_u\, [u\dbr\mid \lambda_u\in\C,\ \sum_{u\in\caS} |\lambda_u|^2<\infty\,\Big\}.}\]
\item[(b)] For $f\in L^2(G,\C)^{\caO}$, $\supp(f)\subseteq G$ is $\sigma$-finite subset.
\end{itemize}
\end{fact}

By $\caB_G(G/\caO)=\caB_G(L^2(G,\C)^{\caO})$ we denote the set of {\it bounded linear operators} on $L^2(G,\C)^{\caO}$
which commute with the left $G$-action. It is well-known that $\caB_G(G/\caO)$
together with the adjoint map $\argu^\ast\colon \caB_G(G/\caO)^{\op}
\to\caB_G(G/\caO)$ and the operator norm $\nor\argu\nor\colon \caB_G(G/\caO)\to\R_0^+$
is a $C^\ast$-algebra (cf.~\cite[\S 1.15]{sakai:cast}).
By definition, $\caB_G(G/\caO)\subseteq\caB(L^2(G,\C)^{\caO})$ is closed with respect to the
weak operator topology (cf.~\cite[Lemma~1.16]{emm:idem}), and thus it is also a 
$W^\ast$- or van Neumann algebra (cf.~\cite[p.~34]{sakai:cast}).


\subsection{The algebra ${}^\caO L^1_\ast(G,\C)^\caO$}
\label{ss:L1star}
For $f\in {}^\caO L^1_\ast(G,\C)^\caO$ (cf.~\eqref{eq:L1ast}), we define 
\begin{equation}
\label{eq:L1star}
\nor f\nor_1=\int_G |f(\omega)|\,\dd \mu_\caO(\omega).
\end{equation}
It is straightforward to verify that $\nor\argu\nor_1\colon{}^\caO L^1_\ast(G,\C)^\caO\to\R_0^+$ is a norm making
$({}^\caO L^1_\ast(G,\C)^\caO,\nor\argu\nor_1)$ a normed complex vector space. It comes equipped with a $\C$-linear map $\eps\colon {}^\caO L^1_\ast(G,\C)^\caO\longrightarrow \C$ given by
\begin{equation}\label{eq:eps}
\eps(f)=\int_G f(\omega)\,\dd\mu_\caO(\omega),
\qquad f\in {}^\caO L^1_\ast(G,\C)^\caO,
\end{equation}
and a skew-linear map $\argu^\ast\colon {}^\caO L^1_\ast(G,\C)^\caO\to {}^\caO L^1_\ast(G,\C)^\caO$ (cf.~\eqref{eq:star}).

The following fact is straightforward 
(cf.~\cite[\S III.5.2, Cor.,~p.263]{bou:top}).

\begin{fact}
\label{fact:suppL1}
Let $G$ be a t.d.l.c.~group, let $\caO$ be a compact open subgroup,
and let $\caR\subseteq G$ be a set of representatives of $\caO\backslash G/\caO$ containing $1\in G$.
\begin{itemize}
\item[(a)] One has 
\[\begin{aligned}
{}^\caO L^1_\ast&(G,\C)^\caO=\Big\{\,f=\sum_{s\in\caR} \nu_s \dbl s\dbr\mid \nu_s\in\C,\\
&\sum_{s\in\caR}  \mu_{\caO}(\caO s\caO)\cdot |\nu_s|<\infty\ \text{and}\  
\sum_{s\in\caR} \mu_{\caO}(\caO s^{-1}\caO)\cdot |\nu_s| <\infty\,\Big\}.
\end{aligned}\]
In particular, $({}^\caO L^1_\ast(G,\C)^\caO,\nor\argu\nor_1)$ is complete.
\item[(b)] For $f\in {}^\caO L^1_\ast(G,\C)^\caO$ its support $\supp(f)$ is $\sigma$-finite.
\end{itemize}
\end{fact}

\begin{rem}
\label{rem:L1star}
For $s\in G$ the set $\caO s\caO$ is compact, and one has decompositions
\begin{equation*}
\label{eq:deco1}
\caO s\caO=\textstyle{\bigsqcup_{1\leq i\leq m(s)} \caO s_i^r=\bigsqcup_{1\leq j\leq n(s)} s_j^l\caO}
\end{equation*}
for suitable elements $s_1^r,\ldots,s_{m(s)}^r, s_1^l,\ldots, s_{n(s)}^l\in G$.
In particular, one has
\begin{equation}
\label{eq:deco2}
\textstyle{\mu_\caO(\caO s\caO)=n(s)
=\sum_{1\leq i\leq m(s)} \Delta(s_i^r)=
\Delta(s)\cdot m(s)}
\end{equation}
and $\dbl s\dbr=\sum_{1\leq j\leq n(s)} [s_j^l\dbr$.
Hence ${}^\caO L^1_\ast(G,\C)^\caO$ is $\C$-subspace of $L^2(G,\C)^\caO$ by Facts~\ref{fact:suppL2} and~\ref{fact:suppL1}.
\end{rem}

For functions $f,h\in {}^\caO L^1_\ast(G,\C)^\caO$ the {\it convolution} $f\ast h$ is defined by
\begin{equation}
\label{eq:conv}
(f\ast h)(x)=\int_{G} f(\omega)\cdot h(\omega^{-1} x)\, \dd \mu_\caO(\omega)
\end{equation}
(cf.~\cite[Thm.~20.10]{hero:ah1}). By Fact~\ref{fact:suppL1}(b), and as $\supp(f\ast h)\subseteq \supp(f)\cdot\supp(h)$, one has
by Fubini's theorem that 
\begin{equation}\label{eq:conv2}
\nor f\ast h\nor_1\leq \nor f\nor_1\cdot \nor h\nor_1,
\end{equation}
i.e., the convolution is continuous, and $f\ast h\in L^1(G,\C)$ (cf.~\cite[Cor.~20.14]{hero:ah1}). Moreover,
\begin{align}
(h^\ast\ast f^\ast) (x)& 
=\int_G \overline{h(\omega^{-1})}\cdot \overline{f(x^{-1}\omega)}\,\dd\mu_\caO(\omega)\notag\\
\intertext{and, substituting $\omega=x\tau$, one obtains (cf.~\eqref{eq:int1})}
&=\int_G \overline{f(\tau)}\cdot \overline{h(\tau^{-1}x^{-1})}\,\dd\mu_\caO(\tau)=(f\ast h)^\ast(x);\label{eq:conv3}
\end{align}
i.e., $(f\ast h)^\ast= h^\ast\ast f^\ast$ and $f\ast h\in L^1_\ast(G,\C)$. 
Finally, for $\omega\in\caO$, one has by \cite[Remark~20.11]{hero:ah1} that
\begin{equation}
\label{eq:astinv}
\begin{aligned}
\omega\ldot(f\ast h)&=(\omega\ldot f)\ast h=f\ast h,\\
(f\ast h)\ldot \omega&=f\ast(h\ldot \omega)=f\ast h,
\end{aligned}
\end{equation}
and $f\ast h\in{}^\caO L^1_\ast(G,\C)^\caO$. This has the following consequence.

\begin{prop}
\label{prop:conv}
Let $\caO$ be a compact open subgroup of the t.d.l.c.\,group $G$.
\begin{itemize}
\item[(a)] $({}^\caO L^1_\ast(G,\C)^\caO,\ast)$ is an associative algebra with 
unit $\dbl 1 \dbr$.
\item[(b)] $\argu^\ast\colon ({}^\caO L^1_\ast(G,\C)^\caO)^{\op}\to {}^\caO L^1_\ast(G,\C)^\caO$ is an involution, i.e., it is a homomorphism of algebras
satisfying $\argu^{\ast\ast}=\iid_{{}^\caO L^1_\ast(G,\C)^\caO}$. 
\item[(c)] $({}^\caO L^1_\ast(G,\C)^\caO,\nor\argu\nor_1,\ast)$ is a Banach $\ast$-algebra. If $G$ is unimodular, one has
$\nor f^\ast\nor_1=\nor f\nor_1$ for all $f\in {}^{\caO}L^1_\ast(G,\C)^{\caO}$.
\item[(d)] $\eps\colon {}^\caO L^1_\ast(G,\C)^\caO\to\C$ is a homomorphism of $\C$-algebras. Moreover, if $G$
is unimodular, then $\eps(f^\ast)=\overline{\eps(f)}$ for all $f\in {}^\caO L^1_\ast(G,\C)^\caO$.
\end{itemize}
\end{prop}

\begin{proof} (a)
Let $f,g,h\in {}^\caO L^1_\ast(G,\C)^\caO$. By Fubini's theorem, one has that
\begin{align}
((f\ast g)\ast h)(z)&=\int_G\Big(
\int_G f(\omega_1)\cdot g(\omega_1^{-1}\omega_2)\,\dd \mu(\omega_1)
\Big)\,h(\omega_2^{-1}z)\,\dd \mu(\omega_2)\notag\\
&=\int_{G\times G} f(\omega_1)\cdot g(\omega_1^{-1}\omega_2)\cdot 
h(\omega_2^{-1} z)\,\dd \mu^{(2)}(\omega_1,\omega_2)\notag\\
&=\int_G f(\omega_1)\cdot\Big(
\int_G  g(\omega_1^{-1}\omega_2)\cdot h(\omega_2^{-1}z)\,\dd \mu(\omega_2)
\Big)\,\dd \mu(\omega_1)\notag\\
\intertext{and substituting $\tau_2=\omega_1^{-1}\omega_2$ yields (cf.~\eqref{eq:int1})}
&=\int_G f(\omega_1)\cdot\Big(
\int_G g(\tau_2)\cdot h(\tau_2^{-1}\omega_1^{-1}z)\,\dd \mu(\tau_2)
\Big)\,\dd\mu(\omega_1)\notag\\
&=(f\ast(g\ast h))(z).\notag
\end{align}
Hence $({}^\caO L^1_\ast(G,\C)^\caO,\ast)$ is an associative algebra. For $f\in {}^\caO L^1_\ast(G,\C)^\caO$,
\begin{align}
(f\ast \dbl 1 \dbr)(z)&=\int_G f(\omega)\cdot 
I_{\caO}(\omega^{-1}z)\,\dd\mu_\caO(\omega)\notag\\
&=\int_G f(\omega)\cdot I_{\caO z^{-1}}(\omega^{-1})\,\dd\mu_\caO(\omega)\notag\\
\intertext{and as $I_{\caO z^{-1}}(\omega^{-1})=I_{z\caO}(\omega)$, this yields}
&=\int_{z\caO} f(\omega)\,\dd\mu_\caO(\omega)=f(z)\cdot\mu_{\caO}(z\caO)=f(z),\notag
\end{align}
i.e., $f\ast \dbl 1 \dbr=f$. Similarly, $\dbl 1 \dbr\ast f=f$.
Part (b) is a direct consequence of \eqref{eq:conv3}. For (c) note that
Fact~\ref{fact:suppL1}(a) implies that $({}^\caO L^1_\ast(G,\C)^\caO,\nor\argu\nor_1)$ is a 
Banach space, and from \eqref{eq:conv2} concludes that
$({}^\caO L^1_\ast(G,\C)^\caO,\nor\argu\nor_1,\ast)$ is a Banach $\ast$-algebra.
If $G$ is unimodular, then $\nor f^\ast\nor_1=\nor f\nor_1$ for all $f\in 
{}^\caO L^1_\ast(G,\C)^\caO$
by \eqref{eq:int3}. Finally, part (d) follows from \eqref{eq:int3}.
\end{proof}


\subsection{The Hecke algebra $\caH(G,\caO)_{\E}$}
\label{ss:hecke}
We define the {\it Hecke $\E$-algebra} of the Hecke pair $(G,\caO)$ by
\begin{equation}
\label{eq:Hec1}
\caH(G,\caO)_\E={}^{\caO}\ca C_c(G,\E)^{\caO}
=\spn_{\E}\{\,\dbl u\dbr\mid u\in G\,\}\subset {}^\caO L^1_\ast(G,\C)^\caO
\end{equation}
(cf. Fact~\ref{fact:perm}(b)). 
Let $\caR\subseteq G$ be a set of representatives of $\caO\backslash G/\caO$ satisfying
\begin{itemize}
\item[(i)] $1\in\caR$, and
\item[(ii)] $r\in\caR$ implies $r^{-1}\in\caR$.
\end{itemize}
By \eqref{eq:deco1}, one has for $s,t\in \caR$ that
\begin{alignat}{2}
(I_{\caO s\caO}\ast I_{\caO t\caO})(z)&=\int_G I_{\caO s\caO}(\omega) \cdot I_{\caO t\caO}(\omega^{-1}z)\,\dd \mu_\caO(\omega)&\notag\\
&=\int_{\caO s\caO} I_{\caO t\caO}(\omega^{-1}z)\, \dd \mu_\caO(\omega)&& \notag\\
&=\sum_{1\leq j\leq n(s)} \int_{s_j^l\caO} I_{\omega\caO t\caO}(z)\,\dd \mu_{\caO}(\omega)&&=\sum_{1\leq j\leq n(s)} I_{s_j^l\caO t\caO}(z).\notag
\end{alignat}
In particular,
\begin{equation}
\label{eq:deco3}
\dbl s\dbr\ast \dbl t\dbr=\sum_{1\leq j\leq n(s)} s_j^l\dbl t\dbr.
\end{equation}
As $\caO s\caO\cdot \caO t\caO\subseteq G$ is compact, there is a finite subset
$\caR(s,t)\subseteq\caR$ such that
\begin{equation}
\label{eq:deco4}
\textstyle{\bigcup_{1\leq j\leq n(s)} s_j^l\caO t\caO=\caO s\caO\cdot \caO t\caO=\bigsqcup_{r\in\caR(s,t)} \caO r\caO.}
\end{equation}
For $s, t\in G$ and $r\in\caR(s,t)$ let $a_{s,t;r}$ be the positive integer satisfying
\begin{equation}
\label{eq:deco5}
a_{s,t;r}=\sum_{1\leq j\leq n(s)} I_{s_j^l\caO t\caO}(r).
\end{equation}
In particular, $0<a_{s,t;r}\leq n(s)$. For $r\in\caR\setminus\caR(s,t)$ put $a_{s,t;r}=0$.
From \eqref{eq:deco3} and \eqref{eq:deco4} one concludes that
\begin{equation}
\label{eq:deco6}
\dbl s\dbr\ast \dbl t\dbr=\sum_{r\in\caR} a_{s,t;r} \dbl r\dbr
\end{equation}
showing that $\caH(G,\caO)_\E$ is an $\E$-subalgebra of $({}^\caO L^1_\ast(G,\C)^\caO,\ast)$.

By construction, one has
\begin{equation}
\label{eq:epsI}
\eps(\dbl s\dbr)=\mu_{\caO}(\caO s\caO)=n(s)=\Delta(s)\cdot m(s)\in\N
\end{equation}
(cf.~\eqref{eq:eps}, \eqref{eq:deco2}), i.e., 
$\eps(\caH(G,\caO)_\E)\subseteq\E$. Furthermore, \eqref{eq:int3} implies also 
\begin{equation}
\label{eq:deco11}
\eps(\dbl s^{-1}\dbr)=\Delta(s^{-1})\cdot \eps(\dbl s\dbr).
\end{equation}


\subsection{The right action of ${}^\caO L^1_\ast(G,\C)^\caO$ on $L^2(G,\C)^{\caO}$}
\label{ss:l12}
Let $\caR\subseteq G$ be a set of representatives of $\caO\backslash G/\caO$ as in the previous subsection. We put
\[\begin{aligned}
\caR^+&=\{\,s\in\caR\mid \Delta(s)>1\,\},\\
\caR^0&=\{\,s\in\caR\mid \Delta(s)=1\,\},\\
\caR^-&=\{\,s\in\caR\mid \Delta(s)<1\,\};\\
\end{aligned}\]
i.e., $\caR=\caR^+\sqcup\caR^0\sqcup\caR^-$.
For an element $f=\sum_{s\in\caR}\nu_s\cdot\dbl s\dbr\in{}^\caO L^1_\ast(G,\C)^\caO$ (cf. Fact~\ref{fact:suppL1}(a)) we put, 
for $x\in\{+,-,0\}$,
\[\textstyle{\nor f\nor_1^x=\sum_{s\in\caR^x}|\nu_s|\cdot \mu_\caO(\caO s\caO).}\]
In particular, $\nor f\nor_1=\nor f\nor_1^+ +\nor f\nor_1^0 + \nor f\nor_1^-$. Therefore,
\begin{equation}\label{eq:decoR3}
\begin{aligned}
&\nor \Delta^{-1/2}\cdot f\nor_1^+ &=&\textstyle{\sum_{s\in\caR^+}|\nu_s|\cdot \mu_\caO(\caO s\caO)\cdot \Delta(s)^{-1/2}}&
\leq& \nor f\nor_1^+,\\
&\nor \Delta^{-1/2}\cdot f\nor_1^0 &=&\textstyle{\sum_{s\in\caR^0}|\nu_s|\cdot 
\mu_\caO(\caO s\caO)}&=&\nor f\nor_1^0,\\
&\nor \Delta^{-1/2}\cdot f\nor_1^- &=&\textstyle{\sum_{s\in\caR^-}|\nu_s|\cdot 
\mu_\caO(\caO s\caO)\cdot \Delta(s)^{-1/2}}
&\leq&\nor \Delta^\ast\cdot f\nor_1^-\\
\end{aligned}
\end{equation}
with $\Delta^\ast\cdot f\in {}^\caO L^1_\ast(G,\C)^\caO$ by \eqref{eq:L1ast}.
In particular, $\Delta^{-1/2}\cdot f\in {}^\caO L^1_\ast(G,\C)^\caO$.
Thus, by \cite[Cor.~20.14]{hero:ah1}, for $\beta\in L^2(G,\C)^{\caO}$ and $f\in {}^\caO L^1_\ast(G,\C)^\caO$ one
concludes that $\beta\ast f\in L^2(G,\C)$, and
\begin{equation}
\label{eq:decoR5}
\nor \beta\ast f\nor_2\leq \nor\beta\nor_2\cdot \nor\Delta^{-1/2}\cdot f\nor_1.
\end{equation}
Moreover, 
\cite[Rem.~20.11(ii)]{hero:ah1} implies that $\beta\ast f\in L^2(G,\C)^{\caO}$.

One has the following property.

\begin{prop}
\label{prop:act}
Let $\caO$\,be a compact open subgroup of the t.d.l.c.\,group $G$.
\begin{itemize}
\item[(a)] Convolution defines a right action of ${}^\caO L^1_\ast(G,\C)^\caO$
on the Hilbert space $L^2(G,\C)^\caO$, and a continuous, injective homomorphism
of Banach $\ast$-algebras $\phi\colon ({}^\caO L^1_\ast(G,\C)^\caO)^{\op}\to \caB_G(G/\caO)$
given by
\[\phi(f)(\beta)=\beta\ast f,
\quad f\in {}^\caO L^1_\ast(G,\C)^\caO,\beta\in L^2(G,\C)^\caO.
\]
\item[(b)] For $f\in \caH(G,\caO)_{\E}$ one has $\phi(f)(\caC_c(G,\E)^{\caO})\subseteq \caC_c(G,\E)^{\caO}$.
In particular, $\phi$ induces an isomorphism
\[\phi_\ast\colon\caH(G,\caO)_{\E}^{\op}\longrightarrow 
\End_{G}(\caC_c(G,\E)^{\caO}).\]
\end{itemize}
\end{prop}

\begin{proof}
(a) By the previously mentioned remark, convolution induces a map
\begin{equation}
\label{eq:act3}
\argu\ast\argu\colon L^2(G,\C)^\caO\times {}^\caO 
L^1_\ast(G,\C)^\caO\longrightarrow L^2(G,\C)^\caO.
\end{equation}
Moreover, the argument used in the proof of Proposition~\ref{prop:conv}(a) can be used verbatim for showing that
\[(\beta\ast f)\ast h=\beta\ast (f\ast h)\]
for $\beta\in L^2(G,\C)^\caO$ and $f,h\in {}^\caO L^1_\ast(G,\C)^\caO$, and
$\beta\ast \dbl 1 \dbr=\beta$, i.e., \eqref{eq:act3} defines a continuous right 
action
of ${}^\caO L^1_\ast(G,\C)^\caO$ on $L^2(G,\C)^\caO$ (cf.~\eqref{eq:decoR3} and \eqref{eq:decoR5}).
In particular, for $f\in {}^\caO L^1_\ast(G,\C)^\caO$, the mapping
\[\argu\ast f\colon L^2(G,\C)^\caO\to L^2(G,\C)^\caO\]
is a bounded linear operator.
By \cite[Rem.~20.11(i)]{hero:ah1}, it commutes with the left $G$-action on $L^2(G,\C)^{\caO}$.
The mapping
\[ [1 \dbr\ast\argu\colon {}^\caO L^1_\ast(G,\C)^\caO \longrightarrow 
L^2(G,\C)^\caO\]
coincides with the (canonical) inclusion ${}^\caO L^1_\ast(G,\C)^\caO \subseteq L^2(G,\C)^\caO$
(cf. Remark~\ref{rem:L1star}). Hence $\phi$ is injective. Let $f\in {}^\caO L^1_\ast(G,\C)^\caO$ and 
$\beta,\gamma\in L^2(G,\C)^\caO$. Then, by \eqref{eq:sequilin}, \eqref{eq:conv} and Fubini's theorem
\begin{align}
\langle \beta\ast f,\gamma\rangle&=\int_G\Big(\int_G \beta(\omega)\cdot f(\omega^{-1}\tau)\,\dd\mu_\caO(\omega)\Big)\cdot
\overline{\gamma(\tau)}\,\dd\mu_\caO(\tau)\notag\\
&=\int_G \beta(\omega)\cdot\Big(\int_G f(\omega^{-1}\tau)\cdot 
\overline{\gamma(\tau)}\,\dd\mu_\caO(\tau)\Big)\,\dd\mu_\caO(\omega)\notag\\
&=\int_G \beta(\omega)\cdot\Big(\int_G \overline{\gamma(\tau)}\cdot\overline{f^\ast(\tau^{-1}\omega)}
\,\dd\mu_\caO(\tau)\Big)\,\dd\mu_\caO(\omega)\notag\\
&=\langle \beta, \gamma\ast f^\ast\rangle.\label{eq:phiast}
\end{align}
In particular, $\phi$ is a morphism of Banach $\ast$-algebras.

\noindent
(b) By (a) and \eqref{eq:deco1}, one has for $u,s\in G$ that
\[ [u\dbr\ast \dbl s\dbr=u\cdot([1\dbr\ast \dbl s\dbr)
=u\cdot \dbl s\dbr
=\sum_{1\leq j\leq n(s)} [us_j^l\dbr.\]
Hence $\phi(f)(\beta)\in \caC_c(G,\E)^\caO$ for $f\in\caH(G,\caO)_{\E}$ and $\beta\in  \caC_c(G,\E)^\caO$
(cf. Fact~\ref{fact:perm}(a)). Thus, by (a), $\phi(f)\in\End_G(\caC_c(G,\E)^{\caO})$, and it suffices to show that
$\phi_\ast$ is surjective. Let $\alpha\in\End_G(\caC_c(G,\E)^{\caO})$.
Then $\alpha([1\dbr)\in\caC_c(G,\E)^{\caO}$. Since $\alpha$ commutes with the left $\caO$-action, and as
$\omega\cdot [1\dbr=[1\dbr$ for all $\omega\in\caO$, one has even 
$\alpha([1\dbr)\in{}^\caO\caC_c(G,\E)^{\caO}=\caH(G,\caO)_{\E}$.
Let $\gamma=\phi_\ast(\alpha([1\dbr))$. Then, by construction,
$(\alpha-\gamma)([1\dbr)=0$. Since $\alpha-\gamma$ commutes with the left $G$-action,
this yields $(\alpha-\gamma)([u\dbr)=0$ for all $u\in G$. Hence, by
Fact~\ref{fact:perm}(a), $\alpha=\gamma$ and this yields the claim.
\end{proof}

The map 
$\iota\colon \caB_G(G/\caO)\to L^2(G,\C)^{\caO}$ given by
$\iota(\alpha)=\alpha([ 1\dbr)$, $\alpha\in\caB_G(G/\caO)$, is a 
continuous homomorphism of topological $\C$-vector spaces.
It has the following property.

\begin{lem}
\label{lem:iota}
The canonical $\C$-linear map $\iota\colon \caB_G(G/\caO)\to L^2(G,\C)^{\caO}$ is injective.
\end{lem}

\begin{proof}
Let $\alpha\in\kernel(\iota)$, i.e., $[1\dbr\in\kernel(\alpha)$. As
$\alpha([u\dbr)=u\cdot\alpha([1\dbr)$, this yields $[u\dbr\in\kernel(\alpha)$ for all $u\in G$.
Thus, by Fact~\ref{fact:suppL2}(a), $\alpha=0$. 
\end{proof}

Note that for $\alpha\in \caB_G(G/\caO)$ one has
\begin{equation}
\label{eq:norhom}
\nor \iota(\alpha)\nor_2=\nor \alpha([1\dbr)\nor_2\leq \nor\alpha\nor.
\end{equation}


\subsection{The C${}^\ast$-Hecke algebra $\bcH(G,\caO)$}
\label{ss:bcH}
By \eqref{eq:Hec1}, Proposition~\ref{prop:act} and Lemma~\ref{lem:iota}, one 
has canonical injections
\begin{equation}
\label{eq:homs}
\xymatrix{
\caH(G,\caO)_\C^{\op}\ar@{^(->}[r]&\big({}^\caO L^1_\ast(G,\C)^{\caO}\big)^{\op}\ar[r]^-{\phi}\ar[r]&\caB_G(G/\caO)\ar[r]^-{\iota}& L^2(G,\C)^{\caO}.}
\end{equation}
We denote by $\bcH(G,\caO)\subseteq\caB_G(G/\caO)$ the closure of 
$\phi(\caH(G,\caO)_\C^{\op})$ in $\caB_G(G/\caO)$. 

Thus, by definition, $\bcH(G,\caO)$ together with the standard operator norm $\nor\argu\nor\colon \bcH(G,\caO)\to\R^+_0$
is a $C^\ast$-subalgebra of $\caB_G(G/\caO)$ which we will call the {\it $C^\ast$-Hecke algebra} associated to $(G,\caO)$. 
\begin{rem}
    Since the weak operator topology is coarser than the operator norm topology, the
map \eqref{eq:homs} remains continuous if $\caB_G(G/\caO)$ is carrying the weak operator topology.
Denote by $W(G,O)$ the $W^\ast$-algebra generated by $\phi(\caH(G,\caO)_\C^{\op})$, which
will be called the {\em $W^\ast$-algebra associated to the Hecke pair $(G, O)$} (cf.~\cite[\S~3.4]{CMW22}, \cite[\S~4.2]{pst:betti})
\end{rem}
By construction, the unital $C^\ast$-algebra $\bcH(G,\caO)$ is acting on the complex Hilbert space $L^2(G,\C)^{\caO}$  (cf. Section~\ref{ss:CHil}).
\begin{thm}\label{thm:1special} Let $G$ be a unimodular t.d.l.c.~group and let $\ca O\subseteq G$ be a compact open subgroup of $G$. The element $[1\dbr=I_\caO$ is a special element of the complex Hilbert space $L^2(G,\C)^{\caO}$.
\end{thm}
\begin{proof} Lemma~\ref{lem:iota} implies that  $[1\dbr$ satisfies the 
condition \ref{eq:S2} of being a special element (cf. Section 
\ref{ss:special}). 
Since $\phi(\ca H(G,\caO)_\C)$ is dense in  $\bcH(G,\caO)$, it suffices to 
verify the 
condition \ref{eq:S3} on the set $\{\phi(\dbl s\dbr)\mid s\in G\}\subseteq 
\bcH(G,\caO)$ (cf.~\eqref{eq:Hec1}). Let $s,t\in G$. Since $G$ is unimodular,  
one computes
\begin{alignat}{2}
\langle [1\dbr \ast \dbl t\dbr,[1\dbr\ast \dbl s\dbr^*\rangle&=\langle \dbl t\dbr, \dbl s^{-1}\dbr\rangle\notag\\
&=\int_G I_{\caO t\caO}(\omega) \cdot I_{\caO s^{-1}\caO}(\omega )\,\dd \mu_\caO(\omega) \notag\\
&=\int_G I_{\caO t\caO}(\omega) \cdot I_{\caO s\caO}(\omega^{-1} )\,\dd \mu_\caO(\omega)\notag\\
&=\int_G I_{\caO t\caO}(\omega^{-1}) \cdot I_{\caO s\caO}(\omega)\cdot\Delta(\omega^{-1})\,\dd \mu_\caO(\omega)\notag\\
&=\int_G I_{\caO t^{-1}\caO}(\omega) \cdot I_{\caO s\caO}(\omega)\,\dd \mu_\caO(\omega)\notag\\
&=\langle \dbl s\dbr, \dbl t^{-1}\dbr\rangle=\langle [1\dbr\ast \dbl s\dbr,[1\dbr\ast \dbl t\dbr^*\rangle\notag
\end{alignat}
(cf. proof of Proposition~\ref{prop:conv}(a)). In particular, the equality~\eqref{eq:phiast} yields 
\begin{alignat}{2}
\langle \phi(\dbl s\dbr\ast\dbl t\dbr)\,.\, [1\dbr, [1\dbr\rangle&=\langle [1\dbr \ast \dbl t\dbr,[1\dbr\ast \dbl s\dbr^*\rangle\notag\\
&=\langle [1\dbr\ast \dbl s\dbr,[1\dbr\ast\dbl t\dbr^*\rangle=\langle \phi(\dbl 
t\dbr\ast\dbl s\dbr)\,.\, [1\dbr, [1\dbr\rangle,\notag
\end{alignat}
i.e., $[1\dbr$ is a special element.
\end{proof}
Therefore, one obtains the trace function 
$$\tau_{[1\dbr}\colon \bcH(G,\caO)\to\C,\quad \tau_{[1\dbr}(h)=\langle h\,.\,[1\dbr,[1\dbr\rangle,\quad h\in \bcH(G,\caO)$$
which is faithful and positive; see Section~\ref{ss:special}.
Moreover, from Theorem~\ref{thm:traceidem} one deduces the following.

\begin{cor}
\label{cor:hecidem}
Let $\caO$ be a compact open subgroup of the unimodular
t.d.l.c.~group $G$, and let $e\in 
\phi(\caH(G,\caO)_\E^{\op})\subseteq\bcH(G,\caO)$ be an idempotent.
Then $\tau_{[1\dbr}(e)\in\E\cap[0,1]$, and $\tau_{[1\dbr}(e)=0$ if and only if $e=0$.
\end{cor}

\begin{proof} 
For $s\in G$, one has
\begin{equation}
\label{eq:cts}
\tau_{[1\dbr}(\phi(\dbl s\dbr))=\langle [1\dbr\ast\dbl s\dbr,[1\dbr\rangle=I_{\caO s\caO}(1)=\begin{cases}1&\text{if\ }s\in\ca O\\ 0&\text{otherwise}\end{cases}.
\end{equation}
In particular, for $h\in \phi(\caH(G,\caO)_\E^{\op})$ one has $\tau_{[1\dbr}(h)\in\E$.
Hence Theorem~\ref{thm:traceidem} yields the claim.
\end{proof}

\begin{rem}
\label{rem:Zform}
By Corollary \ref{cor:hecidem}, any idempotent of $\caH(G,\caO)_\Z$ must be 
$0$ or $1$, i.e., the ring $\caH(G,\caO)_\Z$ is directly indecomposable.
\end{rem}

\begin{rem}
\label{rem:hecke}
The interested reader might have wondered why we have called $\bcH(G,\caO)$ the
$C^\ast$-Hecke algebra associated to $(G,\caO)$ and not the
$C^\ast$-algebra $\caB_G(G/\caO)$. The answer to this question is given
by the fact that, if $G$ is unimodular, then $\tau_{[1\dbr}$ defines
a trace function on $\bcH(G,\caO)$ but we could not find any reason why
this should be true for $\caB_G(G/\caO)$.
Actually, in general, $\alpha_1(\alpha_2([1\dbr))(1)$ is not equal to 
$\alpha_2(\alpha_1([1\dbr))(1)$ for $\alpha_1,\alpha_2\in\caB_G(G/\caO)$.
\end{rem}

\section{Hattori--Stallings rank of  projective discrete  $\QG$-modules}
\label{ss:ratdis}
For a t.d.l.c.~group $G$, a left $G$-set $\Omega$ is said to be {\it discrete} if the pointwise stabilizer $G_\omega$ is an open subgroup of $G$ for every $\omega\in\Omega$. A {\it rational discrete left $G$-module} $M$ is a left $\QG$-module that is discrete as a $G$-set. 
  E.g., for every discrete left $G$-set $\Omega$, the permutation module $\Q[\Omega]$ is a discrete left $\QG$-module.
Denote by $\QGmod$ the abelian category of left $\QG$-modules. The full subcategory $\QGdis$ of $\QGmod$, whose objects are the discrete left $\QG$-modules, turns out to be an abelian category with both enough injectives and projectives (cf.~\cite[Prop.~3.2]{cw:qrat}). In particular one has the following characterization.
\begin{prop}[\protect{\cite[Corollary 3.3]{cw:qrat}}]\label{prop:proj} Let $G$ be a t.d.l.c.~group. A  discrete left $\QG$-module $M$ is projective in $\QGdis$ if, and only if, $M$ is a direct summand of a discrete $\QG$-permutation module $\Q[\Omega]$ for some discrete left $G$-set $\Omega$ with compact pointwise stabilizers.
\end{prop}
\subsection{The rational discrete standard bimodule}
\label{ss:stand}
For a t.d.l.c.~group $G$, let $\CO(G)$ be the set of all compact open subgroups 
of $G$. For every $U,V\in\CO(G)$ let 
\[|U:V|=\frac{|U:U\cap V|}{|V:U\cap V|}\in\Q^+\]
be their commensurability index.

By inverse inclusion, $(\CO(G),\subseteq)$ is clearly a directed set. Following 
\cite{cw:qrat}, whenever $U$ and $V$ are in $\CO(G)$ such that $V\subseteq U$, one 
may define an injective mapping
\begin{equation}\label{eq:tran}
\eta_{U,V}\colon\Q[G/U]\to\Q[G/V],\quad \eta_{U,V}(gU)=\frac{1}{|U:V|}\sum_{r\in\caR}grV,\quad g\in G
\end{equation}
of discrete left $\QG$-permutation modules, where $\caR$ is a set of coset representatives of $V$ in $U$. Since $\eta_{U,W}=\eta_{V,W}\circ\eta_{U,V}$ for $W\subseteq V\subseteq U$ and $W\in\CO(G)$, by \eqref{eq:tran}, one obtains a direct system of projective discrete $\QG$-permutation modules over $\CO(G)$. The discrete left $\QG$-module given by
\begin{equation}\label{eq:BG}
\Bi(G)=\varinjlim_{U,V\in\CO(G)}(\Q[G/U],\eta_{U,V}),
\end{equation}
is called the {\it rational discrete standard bimodule of $G$}. Indeed $\Bi(G)$ inherits a compatible structure of discrete right $\QG$-module via the isomorphisms
\begin{equation}
c_{U,g}\colon\Q[G/U]\to \Q[G/U^g],\quad c_{U,g}(xU)=xgU^g,
\end{equation}
where $U^g=g^{-1}Ug$ for $U\in\CO(G)$ and $g\in G$. Moreover, for every rational discrete $\QG$-module $M$, there exists an isomorphism
\begin{equation}
\theta_M\colon \Bi(G)\otimes_G M\to M,
\end{equation}
which, for all $U\in \CO(G),g\in G$ and $m\in M$, can be described as follows:
\begin{equation}
\theta_M(gU\otimes m)= \frac{1}{|U:U\cap G_m|}\sum_{r\in\ca R}gr\cdot m,
\end{equation}
where $G_m$ is the stabiliser of $m$ and $\ca R$ is a finite set of representatives of left 
cosets of the open set $U\cap G_m$ in the compact subgroup $U$ of $G$ (see 
\cite[\S~4.2]{cw:qrat}).

\medskip
Let $\ca{C}_c(G,\Q)$ be the set of all continuous functions from $G$ to $\Q$ with compact support, where $\Q$ is endowed with the discrete topology. Thus $\caC_c(G,\Q)$ consists of locally constant functions from $G$ to $\Q$ with compact support and
\begin{equation*}\label{eq:CcG}
\caC_c(G,\Q)=\mathrm{span}_\Q\{I_{gU}\mid g\in G, U\in\CO(G)\},
\end{equation*}
where $I_{gU}$ denotes the characteristic function of the coset $gU$ (cf.~\S~\ref{ss:funcomsupp}). 
Since $\caC_c(G,\Q)$ comes equipped with the $G$-actions given by
\begin{equation*}
(g\cdot f)(x)=f(g^{-1}x),\quad (f\cdot g)(x)=f(xg^{-1}),\quad g,x\in G, f\in\caC_c(G,\Q),
\end{equation*}
it is a discrete $\QG$-bimodule. Moreover, the rational discrete left $G$-module $\Bi(G)$ is isomorphic to the left $\Q[G]$-module $\caC_c(G,\Q)$, but the isomorphism is non-canonical: for a fixed compact open subgroup $\caO$ of $G$, one has an isomorphism of rational discrete left $G$-modules
\begin{equation}\label{eq:left iso}
\psi^{(\caO)}\colon\Bi(G)\to\caC_c(G,\Q),
\end{equation}
induced by $\psi^{(\caO)}(gU)=|\caO:U|\cdot I_{gU}$ for any $U\in\CO(G)$.

 If the group $G$ is assumed to be unimodular, then $\psi^{(\caO)}$ turns out to be a (non-canonical) isomorphism of right $\QG$-modules as well (cf.~\cite[Proposition 4.11 (b)]{cw:qrat}).
\subsection{The trace function on $\Bi(G)$}
\label{ss:trace}
Let $G$ be a unimodular t.d.l.c.~group. For a compact open subgroup $\caO$, let 
$\mu_\caO$ denote the Haar measure of $G$ satisfying $\mu_\caO(\caO)=1$. In 
particular, $\mu_\caO=|U:\caO |\cdot\mu_U$ for every $U\in\CO(G)$. We also 
denote by 
\begin{equation}\label{eq:h(G)}
\boh(G)=\Q\cdot\mu_\caO.
\end{equation}
the 1-dimensional $\Q$-vector space describing all Haar measures $\mu_\caO$ where $\caO$ ranges in $\CO(G)$. By the isomorphism \eqref{eq:left iso}, one has a $\Q$-linear map
\begin{equation}\label{eq:trace map}
\tr=\psi^{(\caO)}(\argu)(1)\cdot\mu_{\ca O}\colon\Bi(G)\to\boh(G),
\end{equation}
which is independent on the choice of the compact open subgroup $\caO$, i.e., 
for all $x\in\Bi(G)$ and $\caO, U\in\CO(G)$ one has
\begin{equation}
\tr(x)=\psi^{(\caO)}(x)(1)\cdot\mu_\caO= 
\psi^{(U)}(x)(1)\cdot\mu_{U}\in\boh(G),
\end{equation}
since $\psi^{(U)}=\mu_\caO(U)\cdot\psi^{(\caO)}$ (cf.~\eqref{eq:left iso}).
\begin{prop}\label{prop:tr} 
Let $G$ be a unimodular t.d.l.c.~group. Then
$$\tr(g\cdot u)=\tr(u\cdot g)\quad\text{for all $u\in\Bi(G)$ and $g\in G$}.$$
Moreover, putting $\underline\Bi(G)=\Bi(G)/\langle g\cdot u-u\cdot g\mid u\in\Bi(G), g\in G\rangle_\Q$ the $\Q$-linear map $\tr$ induces the $\Q$-linear map $\underline\tr\colon\underline\Bi(G)\to\boh(G).$
\end{prop}
\begin{proof}
Since $\psi^{(\caO)}$ is an isomorphism of $\QG$-bimodules, it suffices to notice that $(g\cdot f)(1)=f(g^{-1})=(f\cdot g)(1)$ for all $g\in G$ and $f\in\caC_c(G,\Q)$.
\end{proof}

\subsection{The Hattori--Stallings rank}
\label{ss:hatsta}
A discrete $\QG$-module $M$ is said to be {\it finitely generated} if 
there exist finitely many elements $U_1,\ldots,U_n$ of $\CO(G)$ and a 
surjective morphism of $\QG$-modules $\coprod_{i=1}^n\Q[G/U_i] 
\twoheadrightarrow M$. Let $G$ be a unimodular t.d.l.c.~group, and let $P$ be a 
finitely generated projective discrete $\QG$-module. The evaluation map
\begin{equation}
\ev_P\colon\Hom_G(P,\Bi(G))\otimes_\Q P\to\Bi(G),
\end{equation}
given by $ \ev_P(\phi\otimes p)=\phi(p)$, for $\phi\in\Hom_G(P,\Bi(G))$ and 
$p\in P$, induces a mapping $\underline\ev_P\colon\Hom_G(P,\Bi(G))\otimes_G 
P\to\underline\Bi(G)$ such the the diagram
\begin{equation}
\xymatrix{\Hom_G(P,\Bi(G))\otimes_\Q P\ar[r]\ar[d]^{\ev_P}&\Hom_G(P,\Bi(G))\otimes_G P\ar[d]^{\underline\ev_P}\\
\Bi(G)\ar[r]&\underline\Bi(G)}
\end{equation}
commutes, where the horizontal maps are the canonical ones. Since $P$ is finitely generated and projective, by the $\Hom-\otimes$ identity established in \cite[\S 4.3]{cw:qrat}, there exists an isomorphism $\xi_{P,P}\colon \Hom_G(P,\Bi(G))\otimes_G P\to\Hom_G(P,P)$. Thus one defines the map
\begin{equation}
\xymatrix{
\rho_P\colon\Hom_G(P,P)\ar[r]^-{\xi_{P,P}^{-1}}&\Hom_G(P,\Bi(G))\otimes_G P\ar[r]^-{\underline\ev_P}&\underline\Bi(G)\ar[r]^-{\underline\tr}&\boh(G).}
\end{equation}
The value $\trho(P):=\rho_P(\mathrm{id}_P)\in\boh(G)$ will be called the \mbox{{\it Hattori--Stallings rank of $P$}.}
\begin{prop}\label{prop:HSperm} Let $G$ be a unimodular t.d.l.c.~group, and let $\caO\in\CO(G)$. Then $\trho(\Q[G/\caO])=1\cdot \mu_\caO\in\boh(G)$.
\end{prop}
\begin{proof} Let $\eta_\caO\colon\Q[G/\caO]\to\Bi(G)$ be the canonical embedding. One verifies easily that
$$\xi_{\Q[G/\caO],\Q[G/\caO]}(\eta_\caO\otimes_G\caO)=\mathrm{id}_{\Q[G/\caO]},$$
and 
$$\underline\ev_{\Q[G/\caO]} 
(\xi^{-1}_{\Q[G/\caO],\Q[G/\caO]}(\mathrm{id}_{\Q[G/\caO]}))=\underline\caO,$$
where $\underline\caO$ is the image of $\caO$ in $\underline\Bi(G)$. Thus by definition of $\underline\tr$ (cf. Proposition~\ref{prop:tr}), one has
$$\trho(\Q[G/\caO])=\underline\tr(\underline\caO)= I_{\ca O}(1)\cdot \mu_{\ca O}=1\cdot\mu_\caO$$
and this completes the proof.
\end{proof}
\begin{thm}\label{thm:kap}
Let $G$ be a unimodular t.d.l.c.~group and $P$ be a finitely generated projective discrete left $\QG$-module with Hattori--Stallings rank $\trho(P)$. Then $\trho(P)\geq 0$, and $\trho(P)= 0$ if and only if $P=0$.
\end{thm}
\begin{proof}  For some positive integer $n$ and $\ca O\in\CO(G)$, by Proposition~\ref{prop:proj}, one can express $P$ as direct summand of  $\bar P=\coprod_{i=1}^n\Q[G/\ca O]$. Let $P\underset{\pi}{\overset{\iota}{\rightleftarrows}}\bar P$ denote the canonical maps such that $\pi\,\iota=id_P$. By additivity, one has
$$\trho(P)=\rho_P(\mathrm{id}_P)=\rho_{\bar P}(\pi\,\mathrm{id}_P\,\iota).$$ 
Let $\alpha=\pi\,\mathrm{id}_P\,\iota\in\End_G(\bar P)$ and denote by $e_1,\ldots,e_n$ the basis elements of $\bar P$ whose stabilisers coincide with $\ca O$. Given the map $\eta_i\colon\bar P\to\Bi(G)$ defined by $\eta_i(e_j)=\delta_{ij}\cdot \eta_{\ca O}(\caO)$ ($i,j=1,\ldots,n$), one computes
$$\underline\ev_{\bar P}(\xi^{-1}_{\bar P,\bar P}(\alpha))
=\underline\ev_{\bar P}\big(\sum_{i=1}^n\eta_i\otimes_G \alpha(e_i)\big)
=\sum_{i=1}^n\underline\ev_{\bar P}(\eta_i\otimes_G \alpha(e_i))
=\sum_{i=1}^n\eta_i(\alpha(e_i)).$$
Therefore, $\underline\ev_{\bar P}(\xi^{-1}_{\bar P,\bar P}(\alpha))$ is the 
sum of the diagonal elements of the endomorphism $\alpha$ regarded as a matrix
with respect to the basis $e_1,\ldots,e_n$. 
In particular, $\alpha$ is an idempotent matrix with entries in the 
$C^\ast$-Hecke algebra $\bcH(G,\caO)$ associated to $(G,\caO)$ (cf. 
Fact~\ref{fact:perm}(a) and Proposition~\ref{prop:act}(b)). As a consequence, 
$$\trho(P)=\rho_{\bar P}(\alpha)=\sum_{i=1}^n\underline\tr(\eta_i(\alpha(e_i))),$$
and Theorem~\ref{thm:traceidem} applies (cf. Fact~\ref{fact:matrix tr}).
\end{proof}
\begin{rem}\label{rem:HS disc}
It is worth remarking that, whenever $G$ is a discrete group, the rational standard discrete bimodule $\Bi(G)$ coincides with the group algebra $\QG$. In this case, for every finitely generated projective $\QG$-module $P$, the value at $\mathrm{id}_P$ of the map $\underline\ev_{P}\circ\xi^{-1}_{P,P}$ produces the classical Hattori--Stallings rank of $P$, which is not a number rather an element of the abelian group $\QG/[\QG,\QG]$ (cf.~\cite[$\S$~IX.2, Exercise~1]{brown:coh}). The classical $\Q$-valued rank $\rho(P)$ can be then obtained from the Hattori--Stallings rank by applying the homomorphism $\overline{\phantom{x}}\colon\QG/[\QG,\QG]\to\Q$ 
satisfying $\overline 1= 1$ and $\overline \gamma= 0$ for $1\neq\gamma\in\QG$ (cf.~\cite[$\S$~IX.2, Exercise~8]{brown:coh}). In particular, one has
$$\trho(P)=\rho(P)\cdot\mu_{\{1\}},$$
where $\mu_{\{1\}}$ is the counting measure on $G$.
\end{rem}
\section{The Euler characteristic of a unimodular t.d.l.c.~group}
\label{s:euler}

Let $G$ be a unimodular t.d.l.c.~group,
and let $P_1,P_2\in\ob(\QGdis)$ be finitely generated and projective.
Then, by construction, one has that 
$\rho_{P_1\oplus P_2}(\iid_{P_1\oplus P_2})=\rho_{P_1}(\iid_{P_1})+\rho_{P_2}(\iid_{P_2})$, i.e.,
\begin{equation}
\label{eq:euler9}
\trho(P_1\oplus P_2)=\trho(P_1)+\trho(P_2).
\end{equation}
Let $(P_\bullet,\delta_\bullet)$ be a finite projective resolution of $\Q$ in $\QGdis$, i.e.,
there exists $N\geq 0$ such that $P_k=0$ for $k>N$, and $P_k$ is finitely generated for all $k\geq 0$.
From the identity \eqref{eq:euler9} one concludes that the value
\begin{equation}
\label{eq:euler10}
\tchi_G=\textstyle{\sum_{k\geq 0} (-1)^k\, \trho(P_k)\in \boh(G)}
\end{equation}
is independent from the choice of the projective resolution $(P_\bullet,\delta_\bullet)$
(cf.~\cite[Proposition~4.1]{toto:eulhec} for a similar argument). We will call 
$\tchi_G$ the {\it Euler--Poincar\'e characteristic}
of $G$.
\subsection{Compact t.d.l.c.~groups}
\label{sss:coeul}
Let $G$ be a compact t.d.l.c.~group. Then $\tchi_G=\trho(\Q)=1\cdot\mu_{G}$, where $\mu_G$ denotes the probability measure on $G$. The equality $\tchi_G=\trho(\Q)$ is an easy consequence of the fact that the trivial discrete $\QG$-module $\Q$ is projective in $\QGdis$ (cf.~\cite[Proposition~3.7(a)]{cw:qrat}). By Proposition~\ref{prop:HSperm}, the Hattori--Stallings rank $\trho(\Q)$ equals $1\cdot\mu_{G}$ since $\Q$ is isomorphic to the permutation $\QG$-module $\Q[G/G]$. 
\subsection{Fundamental groups of finite graphs of profinite groups}
\label{sss:fundeul}
Let $(\gog,\Lambda)$ be a finite graph of profinite groups (cf.~\cite[$\S$~5.5]{cw:qrat}).
Assume further that the t.d.l.c.~group $\Pi=\pi_1(\gog,\Lambda,x_0)$ is unimodular\footnote{For a given
finite graph of profinite groups it is possible to decide when $\Pi$ is unimodular.
However, the complexity of this problem will depend on the rank of $H^1(|\Lambda|,\Z)$.
E.g., if $\Lambda$ is a finite tree, then $\Pi$ will be unimodular.}.  Then, $\Pi$ acts without inversions on its Bass-Serre tree and, by Theorem~\ref{thm:chi simpl} or \eqref{eq:typeF}, one has
\begin{equation}\label{eq:chi pi}
\tchi_\Pi=\sum_{v\in\euV(\Lambda)} 1\cdot\mu_{{\gog}_v}-\sum_{\eu e\in\ca E^g(\Lambda)} 1\cdot\mu_{\gog_\eu e},
\end{equation}
where  $\euE^g\subseteq\euE$ denotes a set of representatives of the $\Z/2\Z$-action on $\euE$ given by edge inversion. 
In particular, if $(\gog,\Lambda)$ is a finite graph of finite groups one obtains
\begin{equation}
\label{eq:fundeul2}
\tchi_\Pi=\Big(\sum_{v\in\euV(\Lambda)} \frac{1}{|\gog_v|}-\sum_{\eu e\in\euE^g(\Lambda)} \frac{1}{|\gog_\eu e|}\Big)\cdot\mu_{\{1\}}
=\chi_\Pi\cdot\mu_{\{1\}},
\end{equation}
where $\chi_\Pi$ denotes the Euler characteristic of the discrete group $\Pi$
(cf.~\cite[\S II.2.6, Ex.~3]{serre:trees}). 
\begin{prop}
\label{prop:muless}
Let $(\gog,\Lambda)$ be a finite graph of profinite groups, such that
$\Pi=\pi_1(\gog,\Lambda,x_0)$ is unimodular and non-compact. Then
$\tchi_\Pi\leq 0$.
\end{prop}
\begin{proof}
Suppose that one of the open embeddings
$\alpha_\eu e \colon\gog_\eu e\to\gog_{t(\eu e)}$ is surjective.
Then removing the edge $\eu e$ from $\Lambda$,
idetifying $t(\eu e)$ with $o(\eu e)$ and taking the induced
graph of profinite groups $(\gog^\prime,\Lambda^\prime)$,
does not change the fundamental group, i.e.,
one has a topological isomorphism $\pi_1(\gog,\Lambda,x_0)\simeq
\pi_1(\gog^\prime,\Lambda^\prime,x_0^\prime)$.
Thus we may assume that the finite graph of profinite groups has the property
that none of the injections $\alpha_\eu e\colon\gog_\eu e\to\gog_{t(\eu e)}$
is surjective. Since $\Pi$ is not compact, $\Lambda$ cannot be a single vertex.
Thus, as the group is unimodular, $|\gog_{t(\eu e)}:\alpha_\eu e(\gog_{\eu e})|\geq 2$ and  one has 
$\mu_{\gog_{\eu e}}\geq\mu_{\gog_{t(\eu e )}}+\mu_{\gog_{o(\eu e)}}$ for any edge $\eu e \in\ca E(\Lambda)$.
Choosing a maximal subtree of $\Lambda$ then yields the claim.
\end{proof}
\subsection{Uniform lattices}\label{ss:unif lat}
Let $G$ be a t.d.l.c.~group which has a uniform lattice $\Gamma$. In particular, $G$ is unimodular. Denote by $\mu_{\ca O}$ the left-invariant Haar measure of $G$ satisfying $\mu_{\ca O}(\ca O)=1$ for the compact open subgroup $\ca O$ of $G$. By \cite[\S 1.5(10)]{BL01}, given a $G$-set $X$ with compact open stabilisers $G_x$ ($x\in X$), one has
\begin{equation}
\sum_{x\in \Gamma\backslash X}\frac{1}{|\Gamma_{x}|}=\mu^{\ca O}_{G/\Gamma}(G/\Gamma)\cdot \sum_{x\in G\backslash X}\frac{1}{\mu_{\ca O}(G_x)},
\end{equation}
where $\mu^{\ca O}_{G/\Gamma}(G/\Gamma)$ denotes the covolume of $\Gamma$ with respect to  $\mu_{\ca O}$. 
\begin{prop}\label{prop:lattice} 
Let $G$ be a unimodular t.d.l.c.~group of type~$\textup F$. For every uniform lattice $\Gamma$ of $G$ one has
$$\tchi_G=\frac{\chi_\Gamma}{\mu^{\ca O}_{G/\Gamma}(G/\Gamma)}\cdot \mu_{\ca O},$$
where $\chi_\Gamma$ denotes the Euler characteristic of the discrete group $\Gamma$.
\end{prop}
\begin{proof} The claim follows from  \eqref{eq:typeF} and the formula above.
\end{proof}
\begin{rem} We could also prove a similar result for arbitrary t.d.l.c.~groups of type~$\FP$ admitting uniform lattices at the cost
of notational complexity.
\end{rem}
\begin{rem}\label{rem:unif lat}
Following \cite[\S~18.4]{dav:book}, a measure $\mu$ on $G$ is called an \emph{Euler--Poincar\'e measure} if every discrete, cocompact, torsion-free subgroup  $\Gamma\subset G$ has the
following two properties: $\Gamma$ is of type~$\textup{FL}$ and $\chi(\Gamma)= \mu(G/\Gamma)$. In Subsection~\ref{ss:top mod}, Proposition~\ref{prop:lattice} can be used to deduce that $\tilde\chi_G$ is an Euler--Poincar\'e measure for every unimodular t.d.l.c.~group admitting a finite-dimensional topological model, generalising \cite[Theorem~18.4.2]{dav:book}.
\end{rem}

\subsection{t.d.l.c.~groups and $G$-simplicial complexes}\label{ss:simplicial} Let $G$ be a uni\-modular t.d.l.c.~group and $\Sigma$  a contractible\footnote{I.e., the geometric realization of $\Sigma$ is contractible (but we do not claim that $|\Sigma|$ is $G$-homotopic to a point).} $d$-dimensional abstract
simplicial complex acted on by $G$ with compact open stabilisers. Following \cite[(A.8)]{cw:qrat}, for all $0\leq k\leq d$, one defines
$$\widetilde{\Sigma}_k:=\{x_0\wedge\ldots\wedge x_k\mid \{x_0,\ldots,x_k\}\in\Sigma_k\}\subset\Lambda_{k+1}(\Q[\Sigma_0]),$$
where $\Lambda_{k+1}(\Q[\Sigma_0])$ denotes the exterior algebra of the $\Q$-vector space over the set $\Sigma_0$. For every $k$-simplex $\sigma=\{x_0,\ldots,x_k\}$ one has $\underline{\sigma}:=x_0\wedge\ldots\wedge x_k$, and the $G$-orbit $G\cdot\sigma$ can be regarded as a subset of the signed $G$-set $\widetilde{\Sigma}_k$ (cf.~\cite[\S 3.3]{cw:qrat}). The $G$-orbit $G\cdot\sigma$ is called \emph{inverted} if $g\cdot\underline{\sigma}=-\underline{\sigma}$ for some $g\in G$.  Set 
$G_{\pm\sigma}:=\stab_G(\{\underline\sigma,-\underline\sigma\})$  and $G_\sigma:=\stab_G(\{\underline\sigma\})$. Note that both $G_{\pm\sigma}$ and $G_\sigma$ are compact open subgroups of $G$. The orientation character $\textup{sgn}_\sigma\colon G_{\pm\sigma}\to\{\pm1\}$ induces an action on the abelian group $\Q$ (which is trivial if the orbit $G\cdot\sigma$ is not inverted). Denote by  $\Q_\sigma$  the discrete left $\Q[G_{\pm\sigma}]$-module given by the abelian group $\Q$ endowed with the prescribed action.  
\begin{lem}\label{lem:signed}  In the notation above, for every simplex $\sigma$, one has
$$\trho(\idn_{G_{\pm\sigma}}^G(\Q_\sigma))=1\cdot\mu_{G_{\pm\sigma}}=\trho(\idn_{G_{\pm\sigma}}^G(\Q)).$$
\end{lem}
\begin{proof} By Proposition \ref{prop:HSperm}, it suffices to prove that
$$\trho(\idn_{G_{\pm\sigma}}^G(\Q_\sigma))=\begin{cases} 1\cdot\mu_{G_\sigma}-1\cdot\mu_{G_{\pm\sigma}}=1\cdot \mu_{G_{\pm\sigma}},\quad\text{if  $G\cdot\sigma$ is inverted}\\
1\cdot \mu_{G_{\pm\sigma}},\hfill\text{otherwise}.
\end{cases}$$
 If $G\cdot\sigma$ is not inverted, then $\idn_{G_{\pm\sigma}}^G(\Q_\sigma)=\idn_{G_{\pm\sigma}}^G(\Q)\cong\Q[G/G_{\pm\sigma}]$. Hence, Proposition~\ref{prop:HSperm} yields $\trho(\idn_{G_{\pm\sigma}}^G(\Q_\sigma))=1\cdot\mu_{G_{\pm\sigma}}$. Notice that in this case $G_{\pm\sigma}=G_\sigma$ and $1\cdot\mu_{G_\sigma}=1\cdot \mu_{G_{\pm\sigma}}$.

Assume $G\cdot\sigma$ is inverted. Let $C_2$ be the finite group with two elements $\{1,g\}$, and consider the short exact sequence of $\Q[C_2]$-modules
\begin{equation}\label{eq:ses1}
\xymatrix{0\ar[r]&\eu I(C_2)\ar[r]&\Q[C_2]\ar[r]&\Q\ar[r]&0,}
\end{equation}
where $\Q[C_2]\to \Q$ is given by $g\mapsto 1$, and $\eu I(C_2)$ denotes the augmentation ideal. In particular, $\eu I(C_2)$ is $\Q$-generated by the single element $g-1$.  The extension of scalars from $\Q[C_2]$ to $\QG$ given by the orientation character turns \eqref{eq:ses1} into 
\begin{equation*}
\xymatrix{0\ar[r]&\Q_\sigma\ar[r]&\Q[G_{\pm\sigma}/G_\sigma]\ar[r]&\Q\ar[r]&0},
\end{equation*}
which is an exact sequence of discrete left $\Q[G_{\pm\sigma}]$-modules.
 Applying the exact functor $\idn^G_{G_{\pm\sigma}}(\argu)$ yields the short exact sequence 
\begin{equation}\label{eq:ses}
\xymatrix{0\ar[r]&\idn_{G_{\pm\sigma}}^G(\Q_\sigma)\ar[r]&\Q[G/G_\sigma]\ar[r]&\Q[G/G_{\pm\sigma}]\ar[r]&0}
\end{equation}
of discrete left $\Q[G_{\pm\sigma}]$-modules, which splits in  $\QGdis$ since  $\Q[G/G_{\pm\sigma}]$ is projective in $\QGdis$  being  $G_{\pm\sigma}$ compact and open in $G$ (cf.~\cite[Proposition~3.2]{cw:qrat}). By the additivity of $\trho(\argu)$, the rank of $\idn_{G_{\pm\sigma}}^G(\Q_\sigma)$ can be obtained by \eqref{eq:ses}, and so $\trho(\idn_{G_{\pm\sigma}}^G(\Q_\sigma))= 1\cdot\mu_{G_\sigma}-1\cdot\mu_{G_{\pm\sigma}}$ by Proposition~\ref{prop:HSperm}. Finally, since $G_\sigma$  has index 2 in $G_{\pm\sigma}$, $\trho(\idn_{G_{\pm\sigma}}^G(\Q_\sigma))=1\cdot\mu_{G_{\pm\sigma}}$.
\end{proof}
\begin{thm}\label{thm:chi simpl}
Let $G$ be a unimodular t.d.l.c.~group acting on a $d$-dimen\-sional abstract
simplicial complex $\Sigma$ with  compact open  stabilisers.
Let $\Omega_\bullet\subseteq\Sigma_\bullet$ be a set of representatives of the $G$-orbits on the $\bullet$-skeleton $\Sigma_\bullet$, and suppose $|\Omega_\bullet|<\infty$. If $\Sigma$ is contractible
then
\begin{equation}
\label{eq:exchi}
\tchi_G=\sum_{0\leq k\leq d}\quad (-1)^k\Big(\sum_{\omega\in\Omega_k}1\cdot\mu_{G_{\pm\omega}}\Big),
\end{equation}
where $d=\dim(\Sigma)$ and $G_{\pm\omega}:=\mathrm{stab}_G(\{\underline\omega,-\underline\omega\})$.
\end{thm}
\begin{proof} Let $\omega$ be a representative of the $G$-orbits on $\Sigma_k$. The $\Q[G]$-submodule $\QG\cdot\underline\omega$ of the oriented $\QG$-permutation module $C_k(\Sigma):=\mathrm{span}(\tilde\Sigma_k)\subseteq\Lambda_{k+1}(\Q[\Sigma_0])$ is isomorphic to the induced module $\idn_{G_{\pm\sigma}}^G(\Q_\sigma)$ (cf.~\cite[(A.9) and (6.9)]{cw:qrat}). Therefore, the augmented rational chain complex $(C_\bullet(\Sigma),\partial_\bullet,\varepsilon)$ of $\Sigma$ produces a projective resolution of $\Q$ in $\QGdis$, which can be used to compute $\tchi_G$, and so $\eqref{eq:exchi}$. Indeed, the Hattori--Stallings rank of each projective module in $(C_\bullet(\Sigma),\partial_\bullet,\varepsilon)$ is given by Lemma~\ref{lem:signed}.
\end{proof}
\subsection{t.d.l.c.~groups with a finite-dimensional topological model} \label{ss:top mod}
Let $X$ be a CW-complex
on which $G$ acts cellularly and continuously. We require that, for each open cell $e \subset X$ and
each $g \in G$ with $ge \cap e \neq\emptyset$, the multiplication by $g$ is the identity on $e$. Following \cite{CCC20,st:topmod}, $X$ is said to be a
{\em topological model of $G$} if $X$ is contractible, its $G$-stabilisers are open and compact, and the
$G$-action on the $k$-skeleton $X^{(k)}$ is cocompact for every $k \in\mathbb N$. 
Assume that $G$ admits a topological model $X$ of dimension $d$, and denote by $\caF_k$ a set of representatives of conjugacy classes of stabilizers of $k$-cells. As explained in \cite[\S~II.1]{tdieck:trans}, the $k$-skeleton $X^{(k)}$
is built from $X^{(k-1)}$   by (simultaneously) attaching the family of (equivariant) $k$-cells $(G/U\times D^k\mid U\in \ca F_k)$ of type $(G/U\mid U\in\caF_k)$. Since each coset space $G/U$ is discrete,
the augmented cellular chain complex $C_\bullet^{cw}(X)$ with $\Q$-coefficients yields a  projective resolution of $\Q$ in $\QGdis$ (after fixing
an equivariant choice of orientations for the cells). Indeed, each module $C_k^{cw}(X)$ is isomorphic to $\coprod_{U\in\caF_k}\Q[G/U]$.
Assume now that $G$ unimodular. One can compute the Hattori--Stallings rank of $C_k^{cw}(X)$, which is given by  $\sum_{U\in\caF_k} 1\cdot\mu_{U}$ (cf.~Proposition~\ref{prop:HSperm}). Therefore, for every compact open subgroup $\caO$ of $G$, one has
\begin{equation}\label{eq:typeF}
\tchi_G=\sum_{0\leq k\leq d}\, (-1)^k\,\Big(\sum_{U\in\caF_k} \frac{1}{\mu_{\ca O}(U)}\Big)\cdot\mu_{\ca O}.
\end{equation}
\begin{rem}
    The Euler--Poincar\'e characteristic $\tchi_G$ of a unimodular t.d.l.c.~group that admits a topological model of dimension $d$ is an Euler--Poincar\'e measure in the sense of \cite[\S~3.3]{ser:coh} (cf. Remark~\ref{rem:unif lat}).
\end{rem}
\begin{rem}\label{rem:betti} Here we still use the notation introduced above.
In \cite{sauer:betti}, for $k\in\{1,\ldots,d\}$, the author refers to $$\sum_{U\in\caF_k} \frac{1}{\mu_{\ca O}(U)}$$
as the weighted number $c_k(X;G,\mu_\caO)$ of equivariant $k$-cells of $X$ (with respect to $\mu_\caO$) and defines the {\em equivariant Euler characteristic of $X$ (w.r.t. $\mu_\caO$)} as
$$\chi(X;G,\mu_\caO):=\sum_{k=0}^d(-1)^kc_k(X;G,\mu_\caO),$$ which coincides with the rational coefficient in \eqref{eq:typeF}.
For unimodular separable t.d.l.c.~groups admitting a finite-dimensional geometric model, the value $\chi(X;G,\mu_\caO)$ coincides with the alternating sum of the
$L^2$-Betti numbers of $G$ with Haar measure $\mu_\caO$ (cf.~\cite[Theorem~4.6]{sauer:betti}).
\end{rem}
\subsection{Chamber-transitive actions on locally finite regular buildings}\label{ss:build}
First we recall some well-known facts about buildings following~\cite[\S~18.1]{dav:book}, then we  compute the Euler--Poincar\'e characteristic of closed automorphism groups of locally finite regular buildings with transitive actions on the chambers.
\subsubsection{Buildings}  A {\em chamber system over a set $S$} is a set $\Delta$ together with a family of equivalence relations on $\Delta$ indexed by $S$. The elements of $\Delta$ are the {\em chambers} of the system. Two chambers are {\em $s$-equivalent} if they are equivalent with respect to equivalence relation indexed by $s$; they are $s$-adjacent if they are $s$-equivalent but not equal. A {\em gallery} in $\Delta$ is a finite sequence of chambers $(C_0,...,C_k)$ such that,for $1\leq j\leq k$, $C_{j-1}$ is adjacent to $C_j$. The {\em type} of the gallery is the word $s_1\cdots s_k$ such that $C_{j-1}$ is $s_j$-adjacent to $C_j$. 
Given a Coxeter group $(W,S)$, a building  of type $(W , S)$ is a chamber system $\Delta$ over $S$ such that
\begin{itemize}
    \item[(b1)]
 for all $s \in S$, each $s$-equivalence class contains at least two chambers, 
    
 \item[(b2)] there is a function $\delta\colon\Delta\times\Delta\to W$ such that $\delta(C, C') = s_1\cdots s_k\in W$  if, and only if, the chambers $C$ and $C'$ can be joined by a gallery of type $s_1\cdots s_k$ (where $s_1\cdots s_k$ is a reduced word in the alphabet $S$).
 \end{itemize}
A building $\Delta$ of type $(W, S)$ is {\em locally finite} (or, equivalently, it has finite thickness) if, for all $s\in S$, every chamber $C$ has finitely many $s$-adjacent chambers.
The building $\Delta$ is {\em regular} if, for each $s \in S$, each $s$-equivalence class has the same number of elements (which will be denoted by $q_s + 1$ when finite). Given a regular locally finite building $\Delta$, the integers $q_s$ define the {\em thickness vector} ${\bf q}$ of $\Delta$, whose indices range over the conjugacy classes of elements of $S$. 
\subsubsection{Spherical residues} For a Coxeter group $(W,S)$ let $\euS(W,S)$ be the set of all spherical subsets $T$
of $S$, i.e., $T$ generates a finite subgroup $W_T$ of $W$. Let $\Delta$ be a building of type $(W,S)$ and $C$ a chamber. For any subset $T$ of $S$, the {\em $T$-residue of $\Delta$ containing the chamber $C$} is the set $\mathrm{Res}_T(C)$ of all chambers in $\Delta$ that can be connected to $C$ by a gallery of type in $W_T$.
A $T$-residue $\Phi$ is said to be {\em spherical} if $T\in \euS(W,S)$. For example, any single chamber produces a spherical residue with $T=\emptyset$. If $\Delta$ is locally finite, then it follows from (b2) that every spherical residue is finite. In particular, every spherical residue determines a spherical building (cf.~\cite[Corollary~5.30]{ab:build}). Hence one deduces the following:
\begin{prop}[\protect{\cite[Corollary~18.1.18]{dav:book}}]\label{prop:index}
    Let $\Delta$ be a regular building of type $(W,S)$ with thickness vector ${\bf q}$. For a spherical subset $T$ of $S$ and a $T$-residue $\Phi\subset\Delta$ one has that the cardinality of $\Phi$ is given by
    $$\mathrm{Card}(\Phi)=\gamma_{_{W_T,T}}({\bf q}),$$
    where  $\gamma_{_{W_T,T}}({\bf t})$ is the growth series of the Coxeter group $W_T=\langle T\rangle$ with generating system $T$ (cf.~\cite[\S~17.1 and Equation~(17.3)]{dav:book}).
\end{prop}
\subsubsection{Type-preserving automorphism groups of buildings} Let $\Delta$ be a building of type $(W,S)$  and consider the associated chamber graph $(\caC_{\Delta},E\caC_{\Delta})$: the edge-coloured connected simplicial graph with vertex set $$\caC:=\{C\mid\text{$C$ is a chamber of $\Delta$}\},$$ where two vertices $C$ and $D$ are connected by an edge of color $s$ if and only if the chambers $C$ and $D$ are $s$-adjacent in $\Delta$. In particular, if $\Delta$ has finite thickness then $(\caC,E\caC)$ is locally finite. A {\em type-preserving automorphism $h$ of $\Delta$} is a graph automorphism that preserves
the coloring of the edges (cf.~\cite[\S~4]{Kr22}). Denote by $\Aut_0(\Delta)$ the group of type-preserving automorphisms of $\Delta$. The natural topology on the group $\Aut_0(\Delta)$ is the
permutation topology. Namely, the topology with the family $\{\mathrm{Fix}(F) \mid \text{$F$ finite subset of $\caC$}\}$ as a
neighbourhood basis of the identity, where 
$\mathrm{Fix}(F)$ denotes the pointwise stabilizer $\{h\in\Aut_0(\Delta)\mid \forall x\in F\ h(x) = x\}$. Assume that $\Delta$ has finite thickness. It is not difficult to prove that the stabilizer of a chamber is compact and open in the permutation topology. Consequently, $\Aut_0(\Delta)$ is a t.d.l.c.~group; see for example \cite{Kr22}.
\subsubsection{The Davis' realization} 
Given a poset $P$, its geometric realization is defined to be the flag simplicial complex of $P$. Let $\Delta$ be a building of type $(W,S)$ and  denote by $\caR(\Delta)$ the poset of all spherical residues in $\Delta$. The Davis' realization $|\Delta| $ of $\Delta$ is defined to be the geometric realization of the poset
$\caR(\Delta)$, which is a contractible simplicial complex (cf.~\cite[Corollary~18.3.6]{dav:book}). Moreover, if $\Delta$ has finite thickness then $|\Delta| $  is locally finite.
\begin{rem}\label{rem:dav chamber}
    Let $\euS(W,S)$ be the poset of spherical subsets of $S$.
    The map $\caR(\Delta)\to \euS(W,S)$ associating to each spherical $T$-residue its type $T\subseteq S$ induces a map of geometric realizations $\varphi\colon|\Delta| \to |\euS(W,S)|$. Since every simplex in $|\Delta|$ has totally ordered vertices, every simplex in $|\Delta|$ comes with a maximal vertex and a minimal vertex. Given a chamber $C$, $|\Delta|_{\geq \{C\}}$ denotes the set of all simplices of $|\Delta|$ with minimal vertex the $\emptyset$-residue $\{C\}$ (i.e., all the chains of the form $\mathrm{Res}_\emptyset(C)\subseteq \mathrm{Res}_{T_1}(C)\subseteq\cdots\subseteq\mathrm{Res}_{T_k}(C)$) and $\varphi$-restriction  $|\Delta|_{\geq \{C\}}\to |\euS(W,S)|$ is a homeomorphism. For $T\in\euS(W,S)$ we denote by 
\begin{equation}
\label{eq:chains}
|\euS(W,S)|_{\geq T}^{(n)}=\{\,(T_i)_{1\leq i\leq n+1}\in\euS(W,S)^{n+1}\mid T_{1}=T,\, T_i\subsetneq T_{i+1}\,\}    
\end{equation}
the set of all increasing chains of length $n$ in $\euS(W,S)$ starting in $T$.
    The realization of $\bigcup_{n\geq1} \euS_{\geq T}^{(n)}(W,S)$ is the link $\mathrm{Lk}(T)$ of $T$ in the flag complex $|\euS(W,S)|$ (cf.~\cite[Definition A.6.1 and comment right after]{dav:book}).
\end{rem}
Let $h$ be a type-preserving
automorphism of $\Delta$. Since the $W$-distance function $\delta$ is $h$-invariant, i.e., $\delta(h(C), h(D)) = \delta(C, D)$, $h$ induces a simplicial map of $|\Delta| $ (furthermore, as explained in \cite[\S~1.1]{brw:build}, this assignment defines an injective homomorphism from
$\Aut_0(\Delta)$ into the group of homeomorphisms of $|\Delta|$). 
\begin{prop}\label{prop:proper} Let $\Delta$ be a building of type $(W,S)$ of finite thickness. If $G$ is a  closed subgroup of $\Aut_0(\Delta)$, then $G$ acts on $|\Delta| $ with compact open vertex stabilisers. 
\end{prop}
\begin{proof}
   The fact that $G$ acts with compact open stabilisers on the set of spherical residues of $\Delta$ is a consequence of \cite[Lemma~5.13]{Kr22}.
\end{proof}
\begin{cor}\label{cor:chi build}
 Let $\Delta$ be a regular building of type $(W,S)$ with thickness ${q}$, and let $G$ be a unimodular chamber-transitive closed subgroup of $\Aut_0(\Delta)$. Then 
\begin{equation*}\label{eq:chi poincare}
   \tilde\chi_G=\frac{1}{\tgamma_{_{W,S}}({ q})}\cdot\mu_{\caO}
\end{equation*}
where $\caO\in\mathcal{CO}(G)$ is the stabiliser of a chamber and $\tgamma_{_{W,S}}(t)$ denotes the rational function of the Poincar\'e series $\gamma_{_{W,S}}(t)$ associated to the Coxeter group $(W,S)$.
\end{cor}
\begin{proof}
By Proposition~\ref{prop:proper}, the augmented chain complex $(C_\bullet(|\Delta| ),\der_\bullet,\eps)$ is a projective resolution of the trivial left $G$-module $\Q$
in $\QGdis$ (cf. Subsection~\ref{ss:simplicial}). Let $C$ be a chamber of $\Delta$ and denote by $\caO$ its stabiliser in $G$, which is compact and open. Since the action is chamber-transitive, the simplices in $|\Delta|_{\geq \{C\}}$ are representatives of the orbits of the $G$-action on $|\Delta|$ (cf. Remark~\ref{rem:dav chamber}). Thus,
by \eqref{eq:exchi}, one has
\begin{equation*}\label{eq:chi build}
  \tilde\chi_G=\mathlarger{\sum}_{\sigma\subseteq |\Delta|_{\geq \{C\}}}\frac{(-1)^{\mathrm{dim}\sigma}}{\mu_{\caO}(G_\sigma)}\cdot \mu_{\caO}
\end{equation*}
where $G_\sigma$ is the stabilizer of the simplex $\sigma$. Since $|\Delta|_{\geq \{C\}}$ is homeomorphic to the realization  $|\euS(W,S)|$ (cf. Remark~\ref{rem:dav chamber}), one computes
\begin{align}\label{eq:chi 2}
    \tchi_G    =\mathlarger{\sum}_{T\in\euS(W,S)} \mathlarger{\sum}_{k\geq0}(-1)^k\frac{\mathrm{Card}(\euS^{(k)}_{\geq T}(W,S))}{\mu_{\caO}(G_T)}\cdot\mu_{\caO},
\end{align}
where $\euS^{(k)}_{\geq T}(W,S)$ is defined in \eqref{eq:chains} and $G_T$ denotes the setwise stabiliser of the $T$-residue containing the chamber $C$.
As $\sum_{k\geq0} (-1)^k\mathrm{Card}(\euS^{(k)}_{\geq T}(W,S))=1-\chi(Lk(T))$ (cf.~\eqref{eq:chains} and the comment right after), from \eqref{eq:chi 2} one obtains
\begin{align}
\tchi_G&=\mathlarger{\sum}_{T\in\euS(W,S)} \frac{1-\chi(Lk(T))}{\mu_{\caO}(G_T)}\cdot\mu_{\caO}
\intertext{and, as $\mu_{\caO}(G_T)=[G_T\colon \caO]=\gamma_{_{W_T,T}}(q)$ by Proposition~\ref{prop:index}, }
&=\mathlarger{\sum}_{T\in\euS(W,S)} \frac{1-\chi(Lk(T))}{\gamma_{_{W_T,T}}(q)}\cdot\mu_{\caO}=\frac{1}{\tgamma_{_{W,S}}( q)}\cdot\mu_{\caO},
\end{align}
where we used a theorem of R.~Charney and M.~Davis (cf.~\cite[Lemma~2(5)]{cd}) for the last
identity.
\end{proof}
\begin{rem}
The latter result shows that the Euler--Poincar\'e characteristic $\tchi_{G}$ 
of $G$ coincides
with the {\em canonical measure} $G$ as defined in \cite[\S~18.4, p.~311]{dav:book}, which is also an Euler--Poincar\'e measure in the sense of \cite[\S~3.3]{ser:coh}.
\end{rem}
\begin{rem}\label{rem:chi build}
    Let $\bar G$ be a unimodular t.d.l.c.~group acting on $\Delta$ with compact open stabilizers. If the action is chamber-transitive then the statement of Corollary~\ref{cor:chi build} holds also for $\bar G$.
\end{rem}
\begin{rem} Due to the relation between $\tchi_G$ and the $L^2$-betti numbers of $G$ given in Remark~\ref{rem:betti}, the formula \eqref{eq:chi poincare} can also be deduced by \cite[Corollary~3.4]{dymara}.
\end{rem}
\begin{example}[Topological Kac-Moody groups]
\label{ex:EPKM} 
Let $S$ be a finite set. A matrix $A=(a_{s,t})_{s,t\in S}$ with integer entries satisfying
\begin{itemize}
\item[(i)] $a_{s,s}=2$ and $a_{s,t}\leq 0$ for $s\not=t$;
\item[(ii)] $a_{s,t}=0 \Longleftrightarrow a_{t,s}=0$
\end{itemize}
is called a \emph{generalized Cartan matrix} over the set $S$
(cf.~\cite[\S~7.1.1]{remy:kac}).
 A quintuple $(S,A,\Lambda, (c_s)_{s\in S}, (h_s)_{s\in S})$
where $\Lambda$ is a free $\Z$-module of rank $|S|$,
$\Lambda^\vee=\Hom(\Lambda,\Z)$, $c_s\in \Lambda$,
$h_s\in\Lambda^\vee$, satisfying
$\langle c_s,h_t\rangle = a_{t,s}$
is called a \emph{Kac-Moody root datum}. Let $\boG_A$ be the associated
Tits functor, i.e., for every field $\F$, $\boG_A(\F)$ is the $\F$-Kac-Moody group
acting on the twin building $\Delta_\pm$ of type $(W,S)$, where $(W,S)$
is the Coxeter group associated to $A$. 
For a finite field $\F$, $\Delta_+$ is a locally finite simplicial complex, and we will denote by $\hboG_A(\F)$ the topological Kac-Moody group 
of Remy-Ronan type acting on $\Delta_+$ (cf.~\cite{rr:KM}).
By construction and Corollary~\ref{cor:chi build} one obtains the following:
if $q=|\F|$ is the order of $\F$, then
 \begin{equation}
 \label{eq:davis}
 \chi_{\hboG_A(\F)}=\tfrac{1}{\tgamma_{_{W,S}}(q)}\cdot \mu_{\caO},
 \end{equation}
 where $\caO=\stab_{\hboG_A}(C)$ is the stabilizer of a chamber $C$ in $\Delta_+$,
 and $\mu_\caO$ is the left-invariant Haar measure of $\hboG_A(\F)$ satisfying
 $\mu_\caO(\caO)=1$.
\end{example}
\subsubsection{The Euler--Poincar\'e characteristic of split semisimple simply-connected algebraic groups}\label{ss:algebraic} Let $p$ be a prime number, let $F$ be a non-archimedean locally compact field 
of residue characteristic $p$ and let $q$ be the cardinality of the residue 
field $\eu{k}_{F}$ of $F$.
Denote by $G$ the group of $F$-points of a {split} 
semisimple simply-connected algebraic group scheme $\mathbf{G}$ defined over $F$ and let $\Iw$ be an Iwahori subgroup  of $G$ (see 5.2.6 of \cite{bt:2}). Denote by $(\widetilde W,\widetilde S)$ the associated affine Weyl group.
By \eqref{eq:exchi}, the Euler--Poincar\'e characteristic of $G$ is given by
\begin{equation}
    \tchi_G=
\sum_{J\subsetneq\widetilde{S}}  
(-1)^{|\widetilde{S}\setminus J|-1}\mu_{P_J}
,\end{equation}
where $P_J$ is the parahoric subgroup associated to $J\subsetneq\widetilde S$. 
Corollary~\ref{cor:chi build} yields
\[\tchi_G 
=\frac{\mu_\Iw}{\tgamma_{_{  \widetilde W,\widetilde{S}}}(q)}.
\]
 Bott's theorem  (cf.~\cite[\S8.9]{hum:cox})  implies that
\[\tgamma_{_{\widetilde W,\widetilde{S}}}(t) 
=\gamma_{_{W,S}}(t) \prod_{i=1}^{n}\frac{1}{(1-t^{m_i})}
=\prod_{i=1}^{n}\frac{1+t+\cdots+t^{m_i}}{(1-t^{m_i})},\]
where $\{m_1,\dots,m_n\}$ are the 
exponents of $(W,S)$ (which denotes the spherical Weyl group associated to $G$), and one obtains
\begin{equation}
    \tchi_G=(-1)^n\prod_{i=1}^{n}\frac{(q^{m_i}-1)}{1+q+\cdots+q^{m_i}}\mu_\Iw 
    \end{equation}
and deduces that $\tchi_G\in\boh^+(G)$ if $n$ even and $\tchi_G\in\boh^-(G)$ if $n$ odd.


\section{Double coset zeta functions}\label{s:dbz}
Throughout this section $G$ will denote a t.d.l.c.~group, $\caO$ a
compact open subgroup of $G$ and $\mu_\caO$ the Haar measure of $G$ satisfying 
$\mu_\caO(\caO)=1$.

Note that for every $g\in G$ one has
$\mu_\caO(\caO g\caO)=|\caO:\caO\cap{}^g\caO|$.
Let $\caR$ be a set of representatives of $\caO$-double cosets of $G$. 
One says that $G$ satisfies the {\em{double coset property with respect to $\caO$}} if
$\caR(n)=\{r\in \caR\mid \mu_\caO(\caO r\caO)=n\}$
is a finite set for every positive integer $n$. For such a group one defines a formal Dirichlet series by
\begin{equation}
	\label{eq:defzeta}
	\zeta_{_{G,\caO}}(s)=\sum_{n\geq 1} |\caR(n)| 
	n^{-s}=\sum_{r\in\caR}\mu_\caO(\caO 
	r\caO)^{-s}.
\end{equation}
It should be mentioned that $|\caR(n)|$, and so $\zeta_{_{G,\caO}}(s)$, does not depend on the  choice  of $\caR$.

\begin{rem}\label{rem:munorma}
	Let $\Delta$ be the modular function of $G$ (see Section \ref{ss:intfun}).
	Since $\Delta(x)=1$ for every $x\in N_{G}(\caO)$, one has  $\mu_\caO(\caO 
	g\caO)=\mu_\caO(\caO x_1gx_2\caO)$ for every $g\in G$ and $x_1,x_2\in 
	N_{G}(\caO)$.
\end{rem}
\begin{prop}\label{prop:bdcosgrth}
If $G$ satisfies the double coset property with respect to some
compact open subgroup $\caO$,
then $G$ satisfies the double coset property with respect to every compact open subgroup $\caO^\prime$.
\end{prop}

\begin{proof}
Let $\caO,\caU\subseteq G$ be compact open subgroups of $G$,
$\caO\subseteq \caU$. Then $\caO$ is of finite index in $\caU$,
and $n=|\caU:\caO|=\mu(\caU)/\mu(\caO)$. Let $\caS\subset \caU$ be a set of representatives for $\caU/\caO$,
i.e., $\caU=\bigsqcup_{s\in\caS} s\caO$.
By definition, there exists a surjective map
$\alpha_\ast\colon \caO\backslash G\slash \caO\to
\caU\backslash G\slash \caU$ given by
$\alpha_\ast(\caO x\caO)=\caU x\caU$ for $x\in G$. From this one concludes that for $y\in G$ one has
\begin{equation}\label{eq:fibsize}
\alpha_\ast^{-1}(\caU y\caU)=\{\,\caO u_1^{-1}y u_2\caO\mid u_1,u_2\in\caU\,\}
=\{\,\caO s_1^{-1}y s_2\caO\mid s_1,s_2\in \caS\,\}
\end{equation}
In particular, the fibres of $\alpha_\ast$ are finite
of cardinality less or equal to $n^2$.
Let $\caR\subset G$ be a set of representatives
for $\caO\backslash G\slash \caO$, and let
$\caR^\prime\subset G$ be a set of representatives
for $\caU\backslash G\slash \caU$.
By definition, for every $z\in\caR$ there exists a unique
$\alpha(z)\in\caR^\prime$ such that $\caU z\caU=\caU\alpha(z)\caU$.
Hence one has a map $\alpha\colon\caR\to\caR^\prime$.
As $G=\bigcup_{z\in\caR}\caU z\caU$, this map is surjective,
and thus has a section $\beta\colon\caR^\prime\to\caR$, i.e.,
$\alpha\circ\beta=\iid_{\caR^\prime}$.
By construction, $\beta$ is injective. The map $\alpha$ is not 
necessarily injective, but
by \eqref{eq:fibsize}
the cardinality of its fibres is bounded by $n^2$.
For $z\in\caR$ let $k_z\in\N$ such that 
$\mu(\caO z\caO)=k_z\cdot\mu(\caO)$, i.e., there exist elements
$x_1,\ldots,x_{k_z}\in G$ such that $\caO z\caO=\bigsqcup_{1\leq i\leq k_z} x_i\caO$. In particular,
$\caU z\caU=\bigcup_{\substack{1\leq i\leq k_z\\ s\in\caS}}
sx_i\caU$ from which one concludes that
$\mu(\caU z\caU)\leq k_z\cdot n\cdot\mu(\caU)$. This yields
\begin{equation}
\label{eq:cosgrth1}
\frac{k_z}{n}=\frac{\mu(\caO z\caO)}{\mu(\caU)}\leq
\frac{\mu(\caU z\caU)}{\mu(\caU)}\leq k_z\cdot n.
\end{equation}
Hence for $y\in\caR^\prime$ \eqref{eq:cosgrth1} implies that
$\mu(\caO\beta(y)\caO)/\mu(\caO)\leq n\cdot\mu(\caU y\caU)/\mu(\caU)$
and thus
\begin{equation}
\label{eq:cosgrth2}
a_m^\caU\leq\sum_{1\leq j\leq m\cdot n} a_j^\caO,
\end{equation}
where $a_m^\caU=|\{\,y\in\caR^\prime\mid\mu(\caU y\caU)=m\cdot\mu(\caU)\,\}|$, etc. Similarly, for $z\in\caR$
\eqref{eq:cosgrth1} implies $\mu(\caU\alpha(z)\caU)/\mu(\caU)\leq k_z\cdot n$ and thus
\begin{equation}
\label{eq:cosgrth3}
a_m^\caO\leq    n^2\cdot\sum_{1\leq j\leq m\cdot n} a_j^\caU
\end{equation}
Thus from \eqref{eq:cosgrth2} and \eqref{eq:cosgrth3} one concludes that
$G$ has the double coset property with respect to $\caO$ if, and only if,
$G$ has the double coset property with respect to $\caU$.
If $\caO$ and $\caO^\prime$ are arbitrary compact open subgroups
of $G$, then $\caO\cap\caO^\prime$ is a compact open subgroup
which is contained in both of them. Hence the previously mentioned statement implies the claim.
\end{proof}
Proposition~\ref{prop:bdcosgrth} can be seen as a quantitative version of the well-known characterisation of t.d.l.c.~groups with a compact open normal subgroup given by R.~M\"oller in \cite[Corollary~2.14]{moller}.
\begin{prop}
    Let $G$ be a t.d.l.c.~group and $\caO$ a compact open subgroup.
Then $G$ has a compact open normal subgroup if, and only if, there exists $m_{\caO}\in\mathbb N$ such that  $|\caR(n)|=0$ for all $n\geq m_{\caO}$. Moreover, if $G$ has a compact open normal subgroup then $G$ does not have the double coset property.
\end{prop}
\subsection{Weyl-transitive actions on buildings with uniform thickness}\label{ss:weyltransitive}
Let $G$ be group acting Weyl-transitively on a locally finite building $\Delta$ of type $(W,S)$. This is equivalent to say that the action is chamber transitive and that, for every $w\in W$, the stabilizer of each chamber $C$ is transitive on the $w$-sphere
$$\{D\in\Delta\mid\delta(C,D)=w\}.$$
Choose a chamber $C$ and denote by $B$ its stabiliser in $G$. Thus the following are satisfied:
\begin{enumerate}
    \item [(a)]
the set of chambers of $\Delta$ can be identified with $G/B$;
\item[(b)] the $B$-orbits $\{BgB\mid g\in G\}$ on the chambers are in 1-1 correspondence with the elements of $W$;

\item[(c)] there is a  set-theoretic decomposition 
\begin{equation*}\label{eq:bruhat}
    G=\bigsqcup_{w\in W}Bg_wB,
\end{equation*}
which is known as Bruhat-Tits decomposition \cite[\S~6.1.4]{ab:build}. 
\end{enumerate}
In particular, $Bg_wB$ corresponds to the $B$-orbit of a chamber $D$ at distance $w=\delta(C,D)$, which coincides with the $w$-sphere centered in $C$.
\begin{prop}\label{prop:zeta build}
  Let $\Delta$ be a locally finite building of type $(W,S)$ and let $G$ be a t.d.l.c.~group acting Weyl-transitively on $\Delta$ with compact open stabilizers. Assume that $\Delta$ has uniform thickness $q+1$, i.e., all the entries of the thickness vector of $\Delta$ equal $q$. Hence,
   \begin{equation}\label{eq:zeta build}
   \zeta_{G,B}(s)=\sum_{w\in W}\mu_B(B 
	g_wB)^{-s}=\gamma_{_{W,S}}(q^{-s}),
 \end{equation}
 where $B$ denotes the stabiliser of a chamber in $G$. 
	Moreover,
	$\zeta_{G,B}(s)$ defines a  meromorphic function 
	$\tzeta_{G,B}\colon\C\to\bC$ given by 
	$\tzeta_{G,B}(s)=\tgamma_{_{W,S}}(q^{-s})$ for every 
	$s\in\C$, and 
 \begin{equation}\label{eq:chi zeta}
     \tchi(G)=\frac{1}{\tzeta_{G,B}(-1)}\cdot\mu_B.
     \end{equation}
\end{prop}
\begin{proof}
    To prove \eqref{eq:zeta build} it suffices to notice that, by \cite[Lemma~18.1.17]{dav:book}, $\mu_B(Bg_wB)=q^{\ell(w)}$. By \cite[Corollary~5]{brw:build}, $G$ is unimodular and Corollary~\ref{cor:chi build} yields \eqref{eq:chi zeta} (cf. Remark~\ref{rem:chi build}).
\end{proof}

\subsubsection{Spherical standard parabolic subgroups} Given a subset $J\subseteq S$, let
$$P_J:=\bigsqcup_{w\in W_J}Bg_wB.$$
By \cite[Proposition~6.27]{ab:build}, $P_J$ is the stabiliser of a face of the chamber $C$. The subgroups $P_J$ are called {\em standard parabolic subgroups of $G$}. The aim is to investigate $\zeta_{{G,P_J}}(s)$ for $J$ spherical, in which case the subgroup $P_J$ is compact and open in $G$.

\smallskip
For  $J$ and $K$ subsets of $S$
let ${}^J{W}^K$ be the set of representatives of minimal 
length of $({W}_J,{W}_K)$-double cosets 
of ${W}$ (cf.~\cite[Proposition~2.23]{ab:build}), and let 
${W}^K={}^\emptyset{W}^K$ and 
${}^J{W}={}^J{W}^\emptyset$.
By \cite[Lemma 2.25]{ab:build}, if $x\in {}^J{W}^K$ then
${W}_K\cap x^{-1} {W}_J 
x={W}_{K\cap x^{-1} J x}$.
\begin{lem}\label{lemma:respectlength}
	Let $x\in {}^J{W}^K$. The map $c_x: 
	{W}_{K\cap 
		x^{-1}Jx}\longrightarrow {W}_{J\cap x Kx^{-1}}$ 
	defined by $w\longmapsto 
	xwx^{-1}$ is an isomorphism which respects the length, that is  
	$\ell(w)=\ell(c_x(w))$ for every $w\in {W}_{K\cap 
		x^{-1}Jx}$.	
\end{lem}

\begin{proof}
	Since $x^{-1}\in {}^K{W}^J$, one has 
	$c_x({W}_{K\cap 
		x^{-1}Jx})=x({W}_K\cap 
	x^{-1}{W}_Jx)x^{-1}={W}_{J\cap x 
		Kx^{-1}}$ and so the map is well-defined.
	It is a group isomorphism because it is the restriction of the 
	conjugation by $x$.
	Moreover, since $K\cap x^{-1}Jx$ generates ${W}_{K\cap 
		x^{-1}Jx}$ and 
	since $\ell(s)=\ell(c_x(s))=1$ for every $s\in K\cap x^{-1}Jx$, the map 
	$c_x$ respects the length.
\end{proof}

\begin{lem}
	\label{lemma:decwtilde}
	Every $w\in {W}$ can be written in a unique way as 
	$w=yxz$ with 
	$y\in   W_J$, $x\in{}^J  W^K$ and  
	$z\in {}^{K\cap x^{-1}Jx}  W_K$. Moreover, one has 
    $\ell(w)=\ell(y)+\ell(x)+\ell(z)$.
\end{lem}

\begin{proof}
	See the proof of \cite[Theorem 1.2]{Curtis}.
\end{proof}
By \cite[Exercise~6.38]{ab:build}, the following generalization of the Bruhat decomposition holds: for any two standard parabolic subgroups $P_J$ and $P_K$ there is a bijection 
$  W_J\backslash 
  W/  W_K\to P_J\backslash 
G/P_K.$ Therefore, ${}^J  W^K$ is a set 
of representative of $(P_J,P_K)$-double cosets of $G$ and 
\begin{equation}\label{eq:decGPJ}
	G=\bigsqcup_{x\in {}^J  W^K}P_JxP_K
\end{equation}
where  one writes $P_JxP_K$ instead of $P_Jg_xP_K$.
In particular, we recover \eqref{eq:bruhat}  when $J=K=\emptyset$.

\begin{lem}\label{lemma:decPJxPJ}
	For every 
	$x\in{}^J  W^K$ 
	one has \[ P_JxP_K=\bigsqcup_{y\in 
		  W_J}\bigsqcup_{z\in 
		{}^{K\cap x^{-1}Jx}  W_K}B yxz B.\]
\end{lem}

\begin{proof}
	Thanks to (\ref{eq:decGPJ}),
	if $P_JxP_K\cap B  wB\neq\emptyset$ with 
	$  w\in  W$, then 
	$P_JxP_K=P_J  wP_K$ and so $w\in 
	  W_Jx  W_K= 
	  W_Jx{}^{Q}  W_K$ where 
	$Q=K\cap x^{-1}Jx$.
	Hence, $P_JxP_K\subset \bigcup_{y\in 
	  W_J}\bigcup_{z\in 
		{}^{Q}  W_K}B yxzB$ which is a disjoint union 
	by 
	Lemma \ref{lemma:decwtilde} and it is clearly contained in $P_JxP_K$.
\end{proof}
For every subset $X\subset   W$, let 
\[\gamma_{_{X,{S}}}(t):=\sum_{x\in X}t^{\ell(x)}\in\Z[[t]]\]
be the Poincar\'e series of $X$ with respect to the length function $\ell$ defined by ${S}$. 
By definition, it is just a formal power series with non-negative integral 
coefficients, but sometimes it coincides with the Taylor expansion in $0$ of a 
rational function $\tgamma_{_{X,{S}}}(t)\in\C(t)$. If this is the case, we will say that 
$\gamma_{_{X,{S}}}(t)$ defines a rational function $\tgamma_{_{X,{S}}}(t)$. The series 
$\gamma_{_{  W,{S}}}(t)$ defines a rational function 
$\tgamma_{_{  W,{S}}}(t)$ \cite[\S~5.12]{hum:cox}.

Given a spherical subset $J$,  for every $x\in{}^J  W^J$, write $Q(x)=J\cap x^{-1}Jx$.

\begin{lem}\label{lemma:dPJ} Let $J$ be a spherical subset of $S$.
	For every $x\in{}^J  W^J$ one has 
	$\mu_{P_J}(P_JxP_J)
	=|P_J:P_{Q(x)}|q^{\ell(x)}$.
\end{lem}
\begin{proof}
	One has $\mu_{P_J}(P_JxP_J)=\mu_B(P_JxP_J)/|P_J:B|$, which is equal to 
	\[\frac{1}{\gamma_{_{  W_J,J}}(q)}
	\mu_B\Big(\bigsqcup_{y\in 
		  W_J}\bigsqcup_{z\in 
		{}^{Q(x)}  W_J}B yxzB\Big)=
	\frac{1}{\gamma_{_{  W_J,J}}(q)}
	\sum_{y\in   W_J}\sum_{z\in 
		{}^{Q(x)}  W_J}q^{\ell(yxz)}\]
	by Lemma \ref{lemma:decPJxPJ}. 
	Now, by Lemma~\ref{lemma:decwtilde}, one
	has $\ell(yxz)=\ell(y)+\ell(x)+\ell(z)$ and consequently 
	$\mu_{P_J}(P_JxP_J)
	=\gamma_{_{{}^{Q(x)}  W_J,J}}(q)q^{\ell(x)}	
	=|P_J:P_{Q(x)}|q^{\ell(x)}$.
\end{proof}

\begin{prop}\label{prop:pQJJrational}
	For every $Q\subset J$, let 
	\[p_{_{Q,J}}=\{x\in{}^J  W^J\,|\, 
	J\cap x^{-1}Jx=Q\}.\]
	The series $\gamma_{p_{_{Q,J}},{S}}(t)$ defines a rational function 
	$\tgamma_{p_{_{Q,J}},{S}}(t)\in\C(t)$.
\end{prop}
\begin{proof}
	In \cite[Chapter~5]{Edgar}, the author considers the Boolean 
	subalgebra $\mathscr{B}$ of $\mathcal{P}(W)$ generated by 
	\[  W_{x,y,m}=\{w\in 
	  W\,|\,\ell(ywx)=\ell(y)+\ell(w)+\ell(x)-2m\}\]
	where $x,y\in   W$ and $m\in \N$ and he proves that the 
	Poincaré series of elements of $\mathscr{B}$ define rational functions in 
	$\C(t)$.
	In particular, all the following subsets of $W$ belongs to $\mathscr{B}$.
	\begin{align*}
		{}^J  W^J&= 
		\bigcap_{x\in J}  W_{x,1,0}\cap 
		\bigcap_{y\in J}  W_{1,y,0}\\
		\{w\in {}^J  W^J\,|\, wrw^{-1}=s\}&= 
		  W_{r,s,1}\cap {}^J  W^J 
		\text{ 
			for }r,s\in J\\
		\{w\in {}^J  W^J\,|\, wrw^{-1}\in J\}&= 
		\bigsqcup_{s\in J}\{w\in {}^J  W^J\,|\, wrw^{-1}=s\}
		\text{ for }r\in J\\
		\{w\in {}^J  W^J\,|\, wrw^{-1}\notin J\}&= 
		{}^J  W^J\setminus\{w\in 
		{}^J  W^J\,|\, 
		wrw^{-1}\in J\}
		\text{ for }r\in J\\
		\{w\in {}^J  W^J\,|\, wQw^{-1}\subset J\} &= 
		\bigcap_{r\in Q}\{w\in {}^J  W^J\,|\, wrw^{-1}\in 
		J\}
		\text{ for }Q\subset J\\
		\{w\in {}^J  W^J\,|\, w(J\setminus 
		Q)w^{-1}\cap J=\emptyset\}&=
		\bigcap_{r\in J\setminus Q}\{w\in {}^J  W^J\,|\, 
		wrw^{-1}\notin J\}
		\text{ for }Q\subset J
	\end{align*}
	Hence, since 
	$$p_{_{Q,J}}=\{w\in {}^J  W^J\,|\, wQw^{-1}\subset 
	J\}\cap 
	\{w\in {}^J  W^J\,|\, w(J\setminus 
	Q)w^{-1}\cap J=\emptyset\}$$
	belongs to $\mathscr{B}$, the series $\gamma_{p_{_{Q,J}},{S}}(t)$ defines a 
	rational function in $\C(t)$.
\end{proof}

In \cite{Chin4} is explained an algorithm for the explicit calculation of  
$\tgamma_{p_{_{Q,J}},{S}}(t)$.
By Lemma \ref{lemma:dPJ}, whenever $J$ is spherical, one has 
\begin{align}\label{eq:zetaPJ}
	\zeta_{G,P_J}(s)&=
	\sum_{x\in{}^J  W^J}
	|P_J:P_{Q(x)}|^{-s}
	q^{-\ell(x)s}
	=\sum_{Q\subset J}
	|P_J:P_{Q}|^{-s}
	\sum_{\substack{x\in{}^J  W^J\\Q(x)=Q}} 
	q^{-\ell(x)s}\nonumber\\
	&=
	\sum_{Q\subset J}
	|P_J:P_{Q}|^{-s}
	\gamma_{p_{_{Q,J}},{S}}(q^{-s}).
\end{align}
\begin{thm}\label{thm:PJ} Given a spherical set $J\subset S$,
	the series $\zeta_{G,P_J}(s)$
defines a  meromorphic function $\tzeta_{G,P_J}\colon\C\to\bC$ and one has 
	 $\tchi_G=\tzeta_{G,P_J}(-1)^{-1}\mu_{P_J}$.
\end{thm}

\begin{proof}
	By Proposition \ref{prop:pQJJrational}, $\tgamma_{p_{_{Q,J}},{S}}(t)$ 
	is a rational 
	function in $t$ for every $Q\subset J$. Hence, by (\ref{eq:zetaPJ}), the 
	series $\zeta_{G,P_J}(s)$ admits a  meromorphic continuation to $\C$.
	By Lemma \ref{lemma:decwtilde}, one has
	\[\gamma_{_{  W,{S}}}(t)=
	\sum_{x\in {}^J  W^J} 
	\gamma_{_{  W_J,J}}(t)\; 
	\gamma_{_{{}^{Q(x)}  W_J,J}}(t)\;t^{\ell(x)}
	=\gamma_{_{  W_J,J}}(t)\sum_{Q\subset J} 
	\gamma_{_{{}^Q  W_J,J}}(t)\;\gamma_{p_{_{Q,J}},{S}}(t)
 \]
	and then, by (\ref{eq:zetaPJ}), one obtains
	$\tzeta_{G,P_J}(-1)=\frac{\tgamma_{_{  W,{S}}}(q)} 
	{\gamma_{_{  W_J,J}}(q)}$ which is rational in $q$. 
	Furthermore, Corollary~\ref{cor:chi build} yields
	\[\tchi_G=\frac{\mu_B}{\tgamma_{_{  W,{S}}}(q)}=
	\frac{\gamma_{_{  W_J,J}}(q)} 
	{\tgamma_{_{  W,{S}}}(q)}\mu_{P_J}= 
	\tzeta_{G,P_J}(-1)^{-1}\mu_{P_J}.\qedhere\]
\end{proof}
\begin{example}[Topological completions of groups with a root group datum] \label{ex:root} Let $G$ be a group and $(B,\Phi)$ a root datum. The reader is referred to \cite[\S~2.1]{cr} for the definition of a {\em root group datum $\{U_\alpha\}_{\alpha\in\Phi}$ of type $(B,\Phi)$ for $G$.} From any root group datum one can construct the BN-pairs $(B_{\pm},N,S)$ (cf.~\cite[\S~4.1.2]{cr}), and so any group $G$ endowed with a root group datum has two Weyl transitive actions on two buildings $\Delta_{\pm}$, whose kernels coincide with the center of $G$ (cf.~\cite[Corollary~5.12]{cr}). Consequently, as detailed in \cite[\S~6]{cr}, for a group $G$ endowed with a root group datum one can construct two totally disconnected groups $G_+$ and $G_-$ by a process of topological completion. Given $\varepsilon\in\{+,-\}$, the group $G_\varepsilon$ admits the BN-pair $(\widehat B_\varepsilon, N,S)$, where $\widehat B_\varepsilon$ denotes the closure of $B_\varepsilon$ in $G_\varepsilon$. The corresponding building is canonically isomorphic to $\Delta_\varepsilon$, and the kernel of the action of $G_\varepsilon$ on $\Delta_\varepsilon$ is the center $Z(G_\varepsilon)$ (which turns out to be a discrete subgroup of $G_\varepsilon$). When the root groups $U_\alpha$ are finite, it follows that $G_+$ is a t.d.l.c.~groups by \cite[Proposition~6.5]{cr}. Moreover, if in addition $G$ has finite center, then $\widehat B_+$ is compact and Theorem~\ref{thm:PJ} applies to $G_+$ (cf. Remark~\ref{rem:chi build}).
\end{example}

\subsection{Split semisimple and simply-connected algebraic groups} 
Let $p$ be a prime number, let $F$ be a non-archimedean locally compact field 
of residue characteristic $p$ and let $q$ be the cardinality of the residue 
field $\eu{k}_{F}$ of $F$.
Denote by $G$ the group of $F$-points of a split semisimple simply-connected algebraic group $\mathbf{G}$ over $F$, by $\widetilde W$ the associated affine Weyl group and by $\Iw$ an Iwahori subgroup of $G$ (see 5.2.6 of \cite{bt:2}).
The choice of $\Iw$ corresponds to the choice of a 
generating sets $\widetilde{S}$ of $\widetilde{W}$.
We denote by $\ell$ the length function of $\widetilde{W}$ with 
respect to $\widetilde{S}$.

In this context, there exists a locally finite building $\Delta$ of type $(\widetilde{W},\widetilde{S})$ on which $G$  acts Weyl-transitively with compact open stabilizers (cf.\cite{BT72}).
The group $B$ defined in Subsection~\ref{ss:weyltransitive} is the Iwahori subgroup $\Iw$ while the subgroups $P_J$ are called {\em parahoric subgroups} of $G$.
For every $J\subsetneq \widetilde{S}$ we denote by $P_J^1$ the 
pro-$p$-radical of 
$P_J$. 
It is a normal pro-$p$-subgroup of finite index of $P_J$ and so it is in 
$\CO(G)$.
We remark that if $J_1\subset J_2$ then 
$P_{J_2}^1\subset P_{J_1}^1\subset P_{J_1}\subset P_{J_2}$.

\smallskip
Fixed $J\subsetneq\widetilde{S}$, the aim is
to study the series $\zeta_{G,P_J^1}(s)$.
For every $x\in{}^J \widetilde{W}^J$, let $\mathcal{Y}_x$ be a set of 
representatives of left $P_{J\cap 
	xJx^{-1}}$-cosets in $P_J$ and let 
$\mathcal{Z}_x$ be a set of representatives of right 
$P^1_{J\cap x^{-1}Jx}$-cosets in $P_J$.
By Lemmata 3.19 and 3.20 of \cite{Morris}, one has 
\[P_{J\cap xJx^{-1}}=P_J^1(P_J\cap g_xP_Jg_x^{-1}) \quad\text{and}\quad 
P^1_{J\cap x^{-1}Jx}=P_J^1(P_J\cap g_x^{-1}P^1_Jg_x).\]
Then, since $P^1_J$ is normal in $P_J$, one concludes that
\begin{equation}\label{eq:PJ}
	P_J=\bigsqcup_{y\in \mathcal{Y}_x}yP_J^1(P_J\cap g_xP_Jg_x^{-1})
	=\bigsqcup_{z\in \mathcal{Z}_x} (P_J\cap g_x^{-1}P^1_Jg_x)P_J^1z.
\end{equation}

\begin{lem}\label{lemma:doublePJ1}
	Let $x\in{}^J\widetilde{W}^J$. Then
	\[ P_Jg_xP_J=\bigsqcup_{y\in \mathcal{Y}_x}\bigsqcup_{z\in \mathcal{Z}_x}
	P_J^1yg_xzP_J^1.
	\]
\end{lem}

\begin{proof}
	Using (\ref{eq:PJ}) we obtain 
	\begin{align*}
		P_Jg_xP_J&=\bigcup_{y\in \mathcal{Y}_x}yP_J^1(P_J\cap g_xP_Jg_x^{-1})g_x
		P_J
		=\bigcup_{y\in \mathcal{Y}_x}yP_J^1g_xP_J\\
		&=
		\bigcup_{y\in \mathcal{Y}_x}\bigcup_{z\in \mathcal{Z}_x}
		yP_J^1g_x(P_J\cap g_x^{-1}P^1_Jg_x)P_J^1z=
		\bigcup_{y\in \mathcal{Y}_x}\bigcup_{z\in \mathcal{Z}_x}
		yP_J^1g_xP_J^1z
	\end{align*}
	and $yP_J^1g_xP_J^1z=P_J^1yg_xzP_J^1$ since $P^1_J$ is normal in $P_J$.
	Now, suppose that $y_1P_J^1g_xP_J^1z_1=y_2P_J^1g_xP_J^1z_2$ with 
	$y_1,y_2\in\mathcal{Y}_x$ and $z_1,z_2\in\mathcal{Z}_x$.
	We obtain 
	$y_2^{-1}y_1\in P_J\cap P_J^1g_xP_Jg_x^{-1}P_J^1=P_J^1(P_J\cap 
	g_xP_Jg_x^{-1})=P_{J\cap xJx^{-1}}$
	and so $y_1=y_2$.
	This implies $P_J^1g_xP_J^1z_1=P_J^1g_xP_J^1z_2$ and so 
	$z_1z_2^{-1}\in P_J\cap P_J^1g_x^{-1}P^1_Jg_xP_J^1=P_J^1(P_J\cap 
	g_x^{-1}P^1_Jg_x)=P^1_{J\cap x^{-1}Jx}$. Hence, $z_1=z_2$ and then the 
	union is 
	disjoint.
\end{proof}

Lemma \ref{lemma:doublePJ1} and (\ref{eq:decGPJ}) imply that 
\[\Xi_{J}:=\{yg_xz\,|\, x\in {}^J\widetilde{W}^J, y\in\mathcal{Y}_x, 
z\in\mathcal{Z}_x\}\]
is a set of representatives of $P_J^1$-double cosets of $G$. 
By Remark~\ref{rem:munorma}, for every $x\in {}^J\widetilde{W}^J$, 
$y\in\mathcal{Y}_x$ and  
$z\in\mathcal{Z}_x$, one has 
$\mu_{P_J^1}(P_J^1yg_xzP_J^1)=\mu_{P_J^1}(P_J^1g_xP_J^1)$ and then, by Lemma 
\ref{lemma:doublePJ1}, one has
\[\mu_{P_J}(P_Jg_xP_J)=
\sum_{y\in\mathcal{Y}_x}\sum_{z\in\mathcal{Y}_x}
\frac{\mu_{P_J^1}(P_J^1yg_xzP_J^1)}{|P_J:P_J^1|}
=\frac{|\mathcal{Y}_x||\mathcal{Z}_x|}{|P_J:P_J^1|}\mu_{P_J^1}(P_J^1g_xP_J^1).
\]
By Lemma \ref{lemma:respectlength}, one has
$|P_{J\cap xJx^{-1}}:\Iw|=\tgamma_{_{\widetilde{W}_{J\cap xJx^{-1}},J}}(q)=
\gamma_{_{\widetilde{W}_{J\cap x^{-1}Jx},J}}(q)=|P_{J\cap x^{-1}Jx}:\Iw|$
and therefore $|\mathcal{Y}_x|=|P_J:P_{Q(x)}|$. In particular,
\begin{equation}\label{eq:dPJ1}
	\mu_{P_J^1}(P_J^1g_xP_J^1)=
	\frac{|P_J:P_J^1|}{|P_J:P_{Q(x)}||P_J:P^1_{Q(x)}|}\mu_{P_J}(P_JxP_J).
\end{equation}
Hence, one concludes
\begin{align}\label{eq:zetaPJ1a}
	\zeta_{G,P_J^1}(s)&=\sum_{\xi\in\Xi_{J}}\mu_{P_J^1}(P_J^1\xi P_J^1)^{-s}=
	|\mathcal{Y}_x||\mathcal{Z}_x|\sum _{x\in 
		{}^J\widetilde{W}^J}\mu_{P_J^1}(P_J^1g_xP_J^1)^{-s}\nonumber\\
	&=\sum_{x\in{}^J\widetilde{W}^J} 
	\frac{|P_J:P_J^1|^{-s}} 
	{|P_J:P_{Q(x)}|^{-s-1}|P_J:P^1_{Q(x)}|^{-s-1}}\mu_{P_J}(P_JxP_J)^{-s}.
\end{align}
Thus, by Lemma \ref{lemma:dPJ} one obtains
\begin{align}\label{eq:zetaPJ1b}
	\zeta_{G,P_J^1}(s)&=
	\sum_{x\in{}^J\widetilde{W}^J} 
	\frac{|P_J:P_J^1|^{-s}|P_J:P_{Q(x)}|^{-s}} 
	{|P_J:P_{Q(x)}|^{-s-1}|P_J:P^1_{Q(x)}|^{-s-1}}q^{-\ell(x)s}\nonumber\\
	&=\sum_{Q\subset J}
	|P_J:P_{Q}||P_J:P^1_{Q}||P^1_Q:P^1_J|^{-s}
	\sum_{\substack{x\in{}^J\widetilde{W}^J\\Q(x)=Q}}
	q^{-\ell(x)s}\nonumber\\
	&=\sum_{Q\subset J}
	|P_J:P_{Q}||P_J:P^1_{Q}|
	|P^1_Q:P^1_J|^{-s}
	\gamma_{p_{_{Q,J}},\widetilde{S}}(q^{-s}).
\end{align}



\begin{thm}\label{thm:p-rad}
	The series $\zeta_{G,P^1_J}(s)$ defines a meromorphic function 
	$\tzeta_{_{G,P^1_J}}\colon\C\to\bC$, which is a rational function in 
	$q^{-s}$, and one has $\tchi_G=\tzeta_{_{G,P^1_J}}(-1)^{-1}\cdot\mu_{P^1_J}$.
\end{thm}

\begin{proof}
	By Proposition \ref{prop:pQJJrational}, $\tgamma_{p_{_{Q,J}},\widetilde{S}}(t)$ 
	is a rational function in $t$ for every $Q\subset J$. Hence, since 
	$|P^1_Q:P^1_J|$ is a power of $q$,  by 
	(\ref{eq:zetaPJ1b}) the series 
	$\zeta_{G,P^1_J}(s)$ admits a meromorphic continuation to $\C$ which is a 
	rational function in $q^{-s}$.
	Furthermore, by (\ref{eq:zetaPJ1a}) we have
	$\tzeta_{G,P^1_J}(-1)=|P_J:P_J^1|\tzeta_{G,P_J}(-1)$ and by Theorem 
	\ref{thm:PJ} we obtain
	$\tchi_G=\tzeta_{G,P_J}(-1)^{-1}\mu_{P_J}= 
	\tzeta_{G,P^1_J}(-1)^{-1}\mu_{P^1_J}$.
\end{proof}
\begin{ques}
    Let $\Delta$ be a locally finite building of type $(W,S)$ and let $G$ be a Weyl-transitive closed subgroup of $\Aut_0(\Delta)$. Assume that $\Delta$ has uniform thickness $q+1$, i.e., all the entries of the thickness vector of $\Delta$ equal $q$, and choose a chamber $C$. Let $P_J$ be a spherical parabolic subgroup, i.e., $P_J$ is the stabiliser of a face $D$ of the chamber $C$. Denote by $U_J$ the kernel of the action of $P_J$ on the link of $D$. Does the series 
    $\zeta_{G,U_J}(s)$ define a meromorphic function that is a rational function in 
	$q^{-s}$? Additionally, does one have $\tchi_{G}=\tzeta_{G,U_J}(-1)^{-1}\cdot\mu_{U_J}$?
\end{ques}

\printbibliography
\end{document}